\documentclass[12pt,reqno]{amsart}
\usepackage{graphicx}
\usepackage{amsmath}
\usepackage{amssymb}
\usepackage{amscd}
 \usepackage{tikz}
\usepackage{pdfpages}
\usepackage{bbm}
\usepackage{amsthm}
\usepackage{amsfonts}
\usepackage{cite}
\newtheorem{theorem}{Theorem}
\newtheorem{theorema}{Theorem}
\newtheorem{theoremb}{Theorem}
\newtheorem{theoremc}{Theorem}

\newtheorem{rk}[theorema]{Remark}
\newtheorem{lem}[theoremb]{Lemma}
\newtheorem{prop}[theoremc]{Proposition}

\newcommand\bib[1]{\bibitem[#1]{#1}}
\newcommand{\comm}[1]{}
\newcommand\1{{\bf 1}}
\renewcommand\a{\alpha}
\renewcommand\b{\beta}

\newcommand\C{{\mathbb C}}

\newcommand\e{\varepsilon}
\newcommand\ee{\textnormal{e}}

\newcommand\g{\mathfrak{g}}

\newcommand\h{\mathfrak{h}}
\renewcommand\H{\mathbb{H}}
\newcommand\hs{\hskip1pt}

\renewcommand\Im{{\textnormal{\texttt{Im}}}}

\renewcommand\l{\lambda}

\newcommand\La{\Lambda}
\newcommand\m{\mathfrak{m}}

\newcommand\oo{\omega}
\newcommand\op[1]{\mathop{\rm #1}\nolimits}
\newcommand\ot{\otimes}
\newcommand\p{\partial}

\newcommand\R{{\mathbb R}}
\renewcommand\Re{{\textnormal{\texttt{Re}}}}
\renewcommand\t{\tau}
\newcommand\te{\theta}

\newcommand\vp{\varphi}
\newcommand\we{\wedge}

\newcommand\z{\sigma}
\newcommand\Z{{\mathbb Z}}

\begin{document}

 \title[Homogeneous almost complex structures in 6D]{Homogeneous almost complex structures \\
 in dimension 6 with semi-simple isotropy}
 \author{D.\,V. Alekseevsky, B.\,S. Kruglikov, H. Winther}
 \maketitle

 \begin{abstract}
We classify invariant almost complex structures on homogeneous manifolds of dimension 6
with semi-simple isotropy. Those with non-degenerate Nijenhuis tensor have the automorphism group
of dimension either 14 or 9.
An invariant almost complex structure with semi-simple isotropy is necessarily either of
specified 6 homogeneous types or a left-invariant structure on a Lie group.
For integrable invariant almost complex structures we classify all compatible invariant Hermitian structures
on these homogeneous manifolds, indicate their integrability properties (K\"ahler, SNK, SKT)
and mark the other interesting geometric properties (including the Gray-Hervella type).
 \end{abstract}

\bigskip

% \begin{center}
% \framebox{\textsc{Pre-final draft}}
% \end{center}

\bigskip

 \section{Introduction and main results}

Consider an almost complex manifold $(M,J)$, $J^2=-\1$. If $M$ is closed, the automorphisms of $J$ form
a Lie group \cite{BKW}. In general the automorphism group can be infinite-dimensional, but
finite-dimensionality can be guaranteed by additional local (non-degeneracy of the Nijenhuis tensor
\cite{K$_1$}) or global (Kobayashi partial hyperbolicity \cite{Ko,KO}) conditions.

An almost complex structure is integrable if the Nijenhuis tensor $N_J$ vanishes \cite{NW}.
In real dimension 6 (complex dimension 3) non-degeneracy of the Nijenhuis tensor
means that $N_J:\La^2_\C TM\to TM$ is a ($\C$-antilinear) isomorphism. Such structures are important
in applications to critical points of the Hitchin-type functionals and nearly K\"ahler geometry \cite{Br,V}.

As proven in \cite{K$_2$} the local automorphism group $G$ of the structure $J$ on
$M^6$ with non-degenerate $N_J$ has dimension at most 14, and that this bound is
achieved only for the $G_2$-invariant almost complex
structures\footnote{We denote the compact real form of $G_2$ by $G_2^c\subset SO(7)$ and
the split real form with the trivial center by $G_2^*\subset SO(3,4)$.}:
either $G_2^c$-invariant $J$ on $S^6$ or
$G_2^*$-invariant $J$ on its non-compact version $S^{2,4}$.
It is natural to ask what is the next submaximal (= maximal among structures that are not locally $G_2$-invariant)
dimension of the automorphism group.

As one can expect this is still transitive, we confine to (locally) homogeneous structures.
But their classification is cumbersome, and so we restrict further by requiring the isotropy group $H$
to be semi-simple. This generalizes the assumption of \cite{Wo}. In this reference
the almost complex structures on $G/H$ with irreducible isotropy $H$ were classified.
Here we extend this classification in dimension 6. The obtained structures $J$ possess abundant
symmetries ($\dim G\ge9$) by construction.

 \begin{theorem}\label{Thm1}
The only homogeneous almost complex structures on the homogeneous space
$M^6=G/H$ with semi-simple isotropy group $H$
are (up to a covering and a quotient by a discrete central subgroup):
 \begin{enumerate}
\item[(I)\ \,] the homogeneous almost complex structure on $S^6=G_2^c/SU(3)$ or on $S^{2,4}=G_2^*/SU(1,2)$;
\item[(II$_1$)] 4-parametric family on $U(3)/SU(2)$, $U(2,1)/SU(2)$;
\item[(II$_2$)] 4-parametric family on $U(2,1)/SU(1,1)$, \\
                and 2-parametric on $GL(3)/SU(1,1)$;
\item[(III)] left-invariant almost complex structures on a 6D Lie group.
 \end{enumerate}
The tables of the latter structures are given in the Appendix, see also Theorem \ref{Thm3}.
The structures of type II are described in Section \ref{S5}.
 \end{theorem}

 \begin{rk}
As written above, the possible $M^6$ are obtained from the universal covering group $G$ by
additional discrete quotient $M=\Gamma\backslash G/H$. These central subgroups $\Gamma\subset G$ can
be completely described. For instance, instead of $U(3)/SU(2)$ we get $G=\R^1\times SU(3)$,
$H=SU(2)$ and $\Gamma$ is one of the 4 obvious discrete subgroups of the center $Z(G)=\R\times\Z_3$.
Similarly, we get $G=\R\times\widetilde{SU(2,1)}$ or $\R\times\widetilde{SL(3)}$ for other type II cases.
However the invariant almost complex structure $J$ on $M=G/H$ depends in these cases only on 2 parameters, since the
torus gets covered by a cylinder, see the details of the construction in Section \ref{S5}.
 \end{rk}

 \begin{rk}
In dimension 4 all complex representations with semi-simple $H\subset GL(2,\C)$ lead to the flat structure
$G=H\ltimes\C^2$ and so $M^4=\C^2$. If we allow $H$ to be reductive, then 4 new cases appear:
 \begin{itemize}
\item $SU(3)/U(2)=\C P^2$,
\item $SU(2,1)/U(2)=B^4$,
\item $SU(2,1)/U(1,1)=\C P^2\setminus\{pt\}\simeq\C P^1\times\C$
\item $SL(3)/U(1,1)=\R P^2\times\R P^{2\vee}\setminus\mathbb{P}\{(v,p):v\cdot p=0\}\simeq T\R P^2$.
 \end{itemize}
In all these cases $J$ is integrable (complex structure).

Our method can be used to extend the classification to the reductive isotropy $H$ in dimension 6 as well,
which includes
such complex manifolds as $\C P^3=SU(4)/U(3)$, but the tables become rather big.
 \end{rk}

We will also examine the invariant (pseudo-)Riemannian and almost symplectic structures on these
homogeneous 6-manifolds $G/H$, specifying (in the case they are compatible with the almost complex structure)
which of them are Hermitian, K\"ahler, strongly nearly K\"ahler (SNK),
strongly K\"ahler with torsion (SKT) and discuss the Gray-Hervella classes of them.

The SNK condition is closely related with the condition of non-degeneracy of $N_J$
(recall that this means $N_J(\La^2TM)=TM$).
For the Calabi almost complex structure $J$ on $S^6$ it is known that its automorphism group
is the compact real form $G_2^c$. Similarly the split real form $G_2^*$
is the symmetry group of the homogeneous structure $J$ on the pseudo-sphere $S^{2,4}$
(in both cases $\dim\op{Aut}(M,J)=14$).
It turns out that for the other cases of Theorem~\ref{Thm1} with non-degenerate tensor $N_J$
the local symmetries of $J$ (and hence the global ones) are only the obvious ones.

 \begin{theorem}\label{Thm2}
Let $J$ be an invariant almost complex structure on the homogeneous space $M=G/H$
from Theorem \ref{Thm1}. Assume that the Nijenhuis tensor $N_J$ is non-degenerate and
 % (this can happen in any of types I, II, III)
that $J$ is not one of two $G_2$-invariant structures (thus $J$ is of types II or III).
Then the (local and global) automorphisms of $J$ in the connected component of unity
are only those coming from $G$, whence $\dim\op{Aut}(M,J)=9$.
 \end{theorem}

 % As we explain in Section X these latter structures can be obtained by a modification of the construction,
 % resembling the Heisenberg algebra extended from 3 dimensions to 9 dimensions via the cross product.

Some calculations from this work used symbolic packages of Maple; the corresponding worksheets are
available from the authors. % upon request.

 \section{Classification result via representation theory}\label{S2}

Consider a homogeneous manifold $M=G/H$, i.e. a connected manifold $M$ on which a
connected Lie group $G$ acts transitively with the stabilizer $H$ of a point $o\in M$.
We will always assume that $G$ acts effectively on $M$, i.e.
no non-trivial subgroup of $H$ is normal in $G$
(in non-effective case we can quotient by the kernel of the action).

In this case the isotropy representation $j:H\to\op{GL}(T_oM)$ is almost faithful
(has finite kernel) provided the stabilizer group $H$ is reductive (in particular, semi-simple
that is our running assumption). When $M$ has a $G$-invariant almost complex structure $J$ whose
Nijenhuis tensor $N_J$ is non-degenerate, then this is also the case by \cite{K$_2$}.

Let $\g,\h$ be the Lie algebras of the Lie groups $G,H$ and $\m=T_oM$ the model tangent space of $G/H$,
$o=eH$. The isotropy representation makes the space $\m$ into $\h$-module. The
above data ($\h$ subalgebra, $\m$ representation) can be summarized in
the following exact 3-sequence of $\h$-modules
 $$
0\to\h\longrightarrow\g\longrightarrow\m\to0.
 $$
Using our hypothesis that $\h$ is semi-simple yields the splitting of the sequence:
we can find an embedding $\m\subset\g$ as an $H$-invariant complement. Thus
we get the reductive decomposition
 $$
\g=\h+\m,\quad [\h,\h]\subset\h,\ [\h,\m]\subset\m.
 $$
 % Moreover, $j$ is the adjoint representation $\op{Ad}_H|_{\m}$.

Our strategy for classification is to start with the pure algebraic data, and then reconstruct
the Lie groups $G,H$ and the manifold $M$ with its geometry (almost complex structure $J$ etc).

Reconstructing $\g=\h\oplus\m$ from the representation $(\h,\m)$ amounts to the following.
The brackets $\mathcal{B}_\h:\La^2\h\to\h$ and $\mathcal{B}_{\h,\m}:\h\we\m\to\m$ are given by
the Lie algebra structure of $\h$ and the $\h$-module structure of $\m$ respectively.
The only missing ingredient is the bracket $\mathcal{B}_\m:\La^2\m\to\g$,
and it determines the full bracket
 $$
\mathcal{B}=\mathcal{B}_\h+\mathcal{B}_{\h,\m}+\mathcal{B}_\m:
\La^2\g=\La^2\h\oplus(\h\we\m)\oplus\La^2\m\to\g=\h\oplus\m.
 $$

 \begin{lem}\label{coro1}
The Jacobi identity of the resulting bracket $\mathcal{B}:\La^2\g\to\g$, involving an element from $\h$,
is equivalent to $\h$-equivariancy of $\mathcal{B}_\m$.
If $\mathcal{B}_\m$ is $\h$-equivariant, then the bracket $\mathcal{B}$ defines the
Lie algebra structure on $\g$ iff the Jacobi map $\op{Jac}_\m:\La^3\m\to\g$ vanishes, where
 $$
\op{Jac}_\m(x,y,z)=\mathcal{B}(x,\mathcal{B}_\m(y,z))+\mathcal{B}(y,\mathcal{B}_\m(z,x))
+\mathcal{B}(z,\mathcal{B}_\m(x,y)).
 $$
 \end{lem}

The Lie algebra $\g$ defined by such $\mathcal{B}_\m$ is called {\it the Lie algebra extension\/}
of the $\h$-module $\m$.

 \begin{proof}
The Jacobi relation involving 3 elements from $\h$, $\op{Jac}_\h:\La^3\h\to\h$, vanishes as $\h$
is a Lie algebra. The Jacobi relation involving 2 elements from $\h$ and 1 from $\m$ vanishes
as $\m$ is an $\h$-representation. Finally the Jacobi relation involving 1 element from $\h$ and
2 from $\m$ is precisely the equivariancy of the map $\mathcal{B}_\m$.
 \end{proof}

Since $\h$ is semi-simple, the construction of $\mathcal{B}_\m\in\op{Hom}_\h(\La^2\m,\g)$ goes as follows.
Decompose into irreducible $\h$-modules (including the trivial):
$\La^2\m=\oplus\,r_i\cdot\mathfrak{u}_i=\oplus\,(\mathfrak{u}_i\otimes\R^{r_i})$,
$\g=\h\oplus\m=\oplus\,s_i\cdot\mathfrak{u}_i=\oplus\,(\mathfrak{u}_i\otimes\R^{s_i})$.
Then by Schur's lemma
 $$
\op{Hom}_\h(\La^2\m,\g)=\bigoplus\mathfrak{gl}_{\h}(\mathfrak{u_i})\otimes\op{Hom}(\R^{r_i},\R^{s_i})
=\oplus\op{Mat}_{s_i\times r_i},
 $$
where the space $\op{Mat}_{s\times r}$ consists of real, complex or quaternionic $s\times r$ matrices
(the algebra of splitting operators $\mathfrak{gl}_{\h}(\mathfrak{u_i})=\R$, $\C$ or $\H$).

Thus $\h$-equivariancy of $\mathcal{B}_\m$ can be effectively checked via the representation theory.
On the contrary, vanishing of $\op{Jac}_\m(\La^3\m)$ is a set of linear and quadratic relations on $\mathcal{B}_\m$
to be checked directly (representation theory can help here too: as $\op{Jac}_\m$ is $\h$-equivariant,
we can decompose $\La^3\m,\g$ into irreducibles and apply Schur's lemma).

Invariant almost complex structures $J$ on $M$ bijectively corresponds to $\h$-invariant tensors
(here endomorphisms\footnote{Endomorphisms are $\R$-linear and $\h$-equivariant transformations of the module.})
 $$
J\in(\m^*\ot\m)^\h=\op{End}_\h(\m)\ \text{ with }\ J^2=-\1.
 $$
Similarly, invariant pseudo-Riemannian metrics and almost symplectic structures on $M$
are in bijective correspondence with non-degenerate $\h$-invariant tensors
$g\in(S^2\m^*)^\h$ and $\omega\in(\Lambda^2\m^*)^\h$ respectively.

\smallskip

Our aim is to classify 6-dimensional homogeneous manifolds $M=G/H$ with
semisimple $H$, admitting an invariant almost complex structure $J$.
Let $\g=\h+\m$ be the associated reductive decomposition. By effectivity
the isotropy representation $\op{ad}:\h\to\mathfrak{gl}(\m)$ is exact
(due to this all elements of $\g$ act as non-trivial symmetries)
and it preserves the complex structure $J$ on $\m$. Therefore we identify
$\h\subset\mathfrak{gl}(\m,J)\simeq\mathfrak{gl}_3(\C)$.

 \subsection{Classification result}

Our strategy is the following:

\smallskip

1. Enumerate all semi-simple subalgebras $\h\subset\mathfrak{gl}_3(\C)$, hence, all
6-dimensional $\h$-modules $\m$ with an invariant complex structure $J$.

2. Describe all $\h$-equivariant linear maps $\mathcal{B}_\m:\La^2\m\to\g$ by
decomposing the module $\La^2\m$ into irreducible submodules.

3. Compute all Lie algebra extensions $\g$ of the $\h$-module $\m$ by solving the
equations $\op{Jac}_\m=0\in\La^3\m^*\ot\g$ on the parameters in $\mathcal{B}_\m$.

4. Determine the homogeneous almost complex manifolds $M=G/H$ associated with the Lie algebra
$\g=\h+\m$ and the complex structure $J$.

\smallskip

The trivial bracket $\mathcal{B}_\m=0$ defines the semidirect product Lie algebra $\g=\h\ltimes\C^3$
corresponding to the manifold $M=\C^3$ with the standard complex structure and the obvious
action of the semi-direct Lie group $G=H\ltimes\C^3$. We call such structure {\it flat\/}
and exclude them from consideration.

\smallskip

Below we use the following notations.
For a classical simple Lie algebra $\h$ denote by $V$ the standard (tautological) $\h$-module.
It has the natural complex structure if $\h$ is $\mathfrak{sl}_2(\C)$, $\mathfrak{sl}_3(\C)$,
$\mathfrak{su}(3)$ or $\mathfrak{su}(2,1)$.
For $\h=\mathfrak{su}(2)$ we identify the module $V$ with the space $\H$ of quaternions
and the algebra $\h$ with imaginary quaternions $\Im(\H)$ acting from the left.
Similarly for $\h=\mathfrak{su}(1,1)$ we identify the module $V$ with the space $\H_s$ of
split quaternions and the algebra $\h$ with split imaginary quaternions $\Im(\H_s)$ acting from the left.
The space of invariant complex structures on $V$ consists of right multiplications $R_q:x\mapsto xq$ by
a quaternion $q$ with $q^2=-1$. So they are parametrized by the unit sphere
$\mathbb{S}^2\subset\Im(\H)$ in the first case and the unit pseudosphere $\mathbb{S}^{1,1}\subset\Im(\H_s)$
in the second case.

We identify the trivial 2-dimensional representation of $\h$ and a complex structure $J$ with
the standard pair $(\C,i)$.
We denote by $\mathfrak{ad}$ the adjoint representation of the Lie algebra $\h$.
If $\h$ is real, the invariant complex structures on the module
$\mathfrak{ad}^\C=\mathfrak{ad}\oplus\mathfrak{ad}$ are parametrized by
 \begin{equation}\label{J}
J({\bf v},{\bf w})=\bigl(r\,{\bf v}-\tfrac{1+r^2}t\,{\bf w},t\,{\bf v}-r\,{\bf w}\bigr),
 \end{equation}
the same concerns the complexified tautological representation $V^\C$ of $\mathfrak{sl}_3(\R)$.

 \begin{theorem}\label{Thm3}
There are 7 different real semi-simple subalgebras in the complex Lie algebra $\mathfrak{gl}_3(\C)$
(up to conjugation):
 \begin{itemize}
\item $\h=\mathfrak{su}(2)$ or $\mathfrak{su}(1,1)$, representations $V+\C$, $\mathfrak{ad}^\C$;
\item $\h=\mathfrak{sl}_2(\C)$, representations $V+\C$, $\mathfrak{ad}$;
\item $\h=\mathfrak{sl}_3(\R)$, representation $V^\C$;
\item $\h=\mathfrak{su}(3)$ or $\mathfrak{su}(2,1)$ or $\mathfrak{sl}_3(\C)$, representation $V$.
 \end{itemize}
For the cases of $\mathfrak{su}(2)$, $\mathfrak{su}(1,1)$ the possible
Lie algebras $\g$ with the specified representations $\m$ are tabulated in the Appendix.
For $\mathfrak{sl}_2(\C)$ the adjoint representation gives only $\g=\mathfrak{sl}_2(\C)\oplus\mathfrak{sl}_2(\C)$
(so that $M=\op{SL}_2(\C)\oplus\op{SL}_2(\C)/\op{SL}_2^{\op{diag}}(\C)$), while $V+\C$ leads to 2 cases for
$\g$ from the Appendix. For $\mathfrak{su}(3)$ the corresponding $\g$ is the Lie algebra of the exceptional
group $G_2^c$. For $\mathfrak{su}(2,1)$ the corresponding $\g$ is the Lie algebra of the exceptional
group $G_2^*$. The other cases $\mathfrak{sl}_3(\R)$ and $\mathfrak{sl}_3(\C)$
give only the flat structures (in which case $M^6=\C^3$ or its quotient).
 \end{theorem}

The proof of this theorem is a straightforward (but lengthy) calculation,
we sketch it in the next section.

 \subsection{Proof of Theorem \ref{Thm1}}

On the level of Lie algebras Theorem \ref{Thm1} follows instantly from Theorem \ref{Thm3}
and the tables from the Appendix.

The passage to the Lie groups is straightforward for types I and II, because in these cases
we indicate the pair $(G,H)$, and it remains to treat only different discrete quotients.
But for type III we need to establish existence of the Lie group $G$ such that $\h$ exponentiates into
its Lie subgroup. In general for a Lie groups $H$ with a homomorphic embedding of Lie algebras
$\iota:\h=\op{Lie}(H)\subset\g$ there does not exist a Lie group $G$ with $\g=\op{Lie}(G)$
and a homomorphic embedding $H\hookrightarrow G$ having differential $\iota$,
see counter-examples and discussion in \cite{B,GOV}.

In our case, however the Lie functor works nicely. The output of Theorem \ref{Thm3}
yields the structure of Lie algebra on $\g=\h+\m$, and we should consider the cases,
when $\m$ is closed with respect to these brackets, i.e. $[,]_\m^\h=0$ in terms of
splitting of the bracket $\mathcal{B}_\m$. Indeed, in the cases when
$\m\subset\g$ is not a Lie subalgebra (types I and II in Theorem \ref{Thm1})
we have an obvious Lie subgroup $H\subset G$ such that $M=G/H$
is the required homogeneous almost complex manifold.

 \begin{prop}
The pairs $(\g,\h)$ of Lie algebra/subalgebra from the Tables of the Appendix with $\m$ being a Lie
algebra correspond to the pairs $(G,H)$ of Lie group/subgroup of type III in Theorem \ref{Thm1}.
 \end{prop}

This statement concerns the cases A1-A6 from the Tables except for the cases
A1.4 and A3.5, which correspond to type II (case A6 is rather simple and was already
discussed in Theorem \ref{Thm3}).

 \begin{proof}
Let $M$ be the simply connected Lie group corresponding to $\m$ (for the representation $V+\C$
the Lie algebra $\m$ is solvable and so $M\simeq\R^6$ as a manifold, for the representation $\mathfrak{ad}^\C$
the choice of $M$ is obvious). Consider the representation $\rho:\h\to\op{End}(\m)$. Then there exists
a homomorphism $R:H\to GL(\m)$ such that $dR=\rho$. By virtue of Proposition 4.2 of
\cite[Chapter 2]{VO} the semi-direct product $G=H\ltimes_R M$ is the desired simply-connected Lie
group. The main idea of this approach follows Cartan's proof of the third Lie theorem \cite{C}.

Another proof is based on the Palais' criterion \cite{P} for a transformation group to be a Lie group.
Namely, $M$ acts on itself by left translations and $H\subset\op{Diff}(M)$ is a closed subgroup
(as it is the stabilizer of the unity and closed in $GL(T_eM)$). Both $M$ and $H$ generate a closed
subgroup in the group of diffeomorphisms, which is a Lie group $G$ with $\op{Lie}(G)=\g=\h\ltimes\m$.
 \end{proof}

Thus, when $\m\subset\g$ is a Lie subalgebra we also get representation $M=G/H$ and
this finishes the proof of Theorem \ref{Thm1}.

 \section{Proof of the classification result}\label{S2+}

Representation part of Theorem \ref{Thm3} (list of 7 cases) is obvious from the general
theory of representations of semi-simple Lie algebras. The hard part is the
reconstruction of possible Lie brackets on $\g=\h+\m$.

Let us consider the first of the cases, when $\h=\mathfrak{su}(2)$ and the isotropy as a complex
representation has type $\m=V+\C$. We can identify $\h$ with the Lie algebra of imaginary
quaternions $\Im(\H)$ and $V$ with the left $\h$-module $\H$. The endomorphism
ring of $V$ is the algebra $\H$ acting from the right (we will use this freedom to change the basis in $V$).
The following easy claim will be used repeatedly.

 \begin{lem}\label{Le2}
The $\h$-invariant complex structure $J$ on the module $\m=V+\C$ is given by
the formula
 $$
J(x,\eta)=(xq,i\eta),\quad x\in V=\H,\eta\in\C,
 $$
where $q\in\mathbb{S}^2\subset\Im(\H)$ is a unit imaginary quaternion.
 \end{lem}
 % We will denote $J$ by the symbol $(q,i)$ and remember its action from the right.
It is possible to fix $q$ to be equal to $i\in\H$ by an endomorphism, but as noticed above we use
this freedom to simplify the brackets. On the other hand, though on the 2-dimensional trivial
$\h$-module $\R^2$ there is a 2-parametric family of complex structures, we fix one
(equal to $i$, turning this submodule into $\C$) as this freedom does not help
to simplify the brackets.

 % The invariant complex structure on $V$ is thus an element $q\in\H$ with $q^2=-1$, so
 % it belongs to the unit sphere $q\in\mathbb{S}^2\subset\Im(\H)$. On the submodule $\C$
 % the complex structure is unique up to sign and we choose it to be $i$. Thus the complex
 % structure on $\m=V+\C$ is given by
 %  $$
 % J(x,\ee)=(xq,i\ee),\quad x\in V,\ee\in\C.
 %  $$

The first task in constructing the map $\mathcal{B}_\m$ is to decompose the $\h$-module
$\La^2\m=\La^2(V+\C)=\La^2V+V\ot\C+\La^2\C$ into irreducibles.
From the representation viewpoint (without complex structure $J$) $\C$ is the trivial
2-dimensional real module $\R^2$, so $\La^2\C\simeq\R^1$ and $V\ot\C=V+V$.

 \begin{lem}
As an $\h$-module $\La^2V=\mathfrak{ad}+\R^3$.
 \end{lem}

 \begin{proof}
Let us give two proofs, one via the representation theory of simple Lie algebras and the other
straightforward.

The complexified Lie algebra $\h^\C=\mathfrak{sl}(2,\C)$ has the same standard representation $V$
(it is not absolutely irreducible). Changing the real form to $\h'=\mathfrak{su}(1,1)\simeq\mathfrak{sl}(2,\R)$ we obtain by the highest weight decomposition
$V=W+W$, where $W$ is the standard representation of $\mathfrak{sl}(2,\R)$. Now
$\La^2(W+W)=\La^2(W)\ot\R^2+W\ot W=S^2W+\R^3=\mathfrak{ad}+\R^3$. This induces the above decomposition.

A more direct proof is as follows.
Let us identify $V\simeq V^*$ using the $\h$-invariant metric on $V=\H$: $g(x,y)=\Re(x\bar{y})$,
$x,y\in V$. Consider the 2-forms $\oo_b,\oo^b\in\La^2V^*$ given by
 $$
\oo_b(x,y)=\Re(xb\bar{y}),\quad \oo^b(x,y)=\Re(x\bar{y}b),\qquad b\in\Im(\H)
 $$
(check both are skew-symmetric!).

The group $H=SU(2)\simeq\mathbb{S}^2\subset\Im(\H)$ acts on them as follows ($q\in H$)
 \begin{gather*}
\oo_b(qx,qy)=\Re(qxb\,\overline{qy})=\Re(qxb\bar{y}\bar{q})=\Re(xb\bar{y}|q|^2)=\oo_b(x,y),\\
\oo^b(qx,qy)=\Re(qx\,\overline{qy}\,b)=\Re(qx\bar{y}\bar{q}b)=\Re(x\bar{y}q^{-1}bq)=\oo_{\op{Ad}_q^{-1}b}(x,y).
 \end{gather*}
Consequently 6-dimensional $\La^2V^*$ has two 3-dimensional submodules $\{\oo^b\}$ and $\{\oo_b\}$.
By the above the first of them has type $\mathfrak{ad}$ and the second is trivial $\R^3$.
Thus they do not intersect and so span the whole $\La^2V^*$. The lemma is proved.
 \end{proof}

 \begin{lem}
Let $\g=\mathfrak{ad}+V+\C$ be a Lie algebra extension of the $\h$-module $V+\C$. Then
the brackets on $\m$ are
 $$
[V,V]\subset\mathfrak{ad}+\C,\ \ [V,\C]\subset V,\ \ [\C,\C]\subset\C.
 $$
 \end{lem}

Here from the naked representation theory viewpoint $\C=\R^2$, but we shall use the structure $i$ on it.
The lemma has the following implications.

 \begin{itemize}
 \item $\C$ is a Lie subalgebra of $\g$ and it is solvable: $\La^2\C\mapsto\R\subset\C$.
There exists a $J$-adapted basis $\ee,i\ee$ of $\C$ such that $[\ee,i\ee]=\e\ee$ for some
$\e\in\{0,1\}$ (for $\e=1$ this determines $\ee$ uniquely).
 \item Since $\C$ is the trivial $\h$-module, the bracket $V\ot\C\to V$ is composed of
endomorphisms, i.e. $[x,\eta]=A_\eta(x)$, where $\eta\mapsto-A_\eta$ is the homomorphism
of Lie algebras $\C\to\op{Lie}(\op{End}_\R(V))=\R\oplus\h\simeq\mathfrak{u}(2)$.
Since the latter has no non-trivial solvable subalgebras, the homomorphism is not injective
when the subalgebra $\C$ is not abelian ($\e=1$ $\Rightarrow$ $A_\ee=0$).
\item The $V$-bracket $\La^2V\to\g$ for some $\l\in\R$, $b,b'\in\Im(\H)$ equals
 $$
[x,y]=\l\cdot\Im(x\bar{y})+\bigl(\oo_b(x,y)\ee+\oo_c(x,y)i\ee\bigr).
 $$
 \end{itemize}
Elaborating upon Lemma \ref{coro1} with these choices $\g$ is a Lie algebra iff the
$V$-Jacobi identity holds -- $\op{Jac}_\m:\La^3V\to\g$ is zero and in addition
 \begin{gather}
\oo_b(A_{i\ee}x,y)+\oo_b(x,A_{i\ee}y)=\e\oo_b(x,y), \notag\\
\oo_b(A_{\ee}x,y)+\oo_b(x,A_{\ee}y)=-\e\oo_c(x,y), \label{Q}\\
\oo_c(A_{\eta}x,y)+\oo_c(x,A_{\eta}y)=0,\quad \eta\in\C \notag.
 \end{gather}

It is also important to notice that since $\H$ is a division algebra, then every nonzero operator
$A_\eta\in\op{End}_\R(V)=\H$ is invertible. Now we consider the following possibilities
(if $\l=0$, then $\m$ is a Lie algebra).

\medskip

(A1.1) $\l=0$, $[V,V]\ne0$ and the subalgebra $A_\C\subset\H$ is nonzero.
We claim that the map $A$ has a kernel. Indeed, for $\e=1$ we have $A_\ee=0$.
For $\e=0$ denoting $A_\eta(x)=xq$, $q\in\H$, then (\ref{Q}) implies
$\Re(x(qd+d\bar{q})\bar{y})=0$, where $d$ is any linear combination of $b,c\in\Im(\H)$.
As $b,c$ are not simultaneously zero, this yields a kernel, which can be accommodated into $\ee$
(using $GL(1,\C)$-freedom of change of basis in the $\C$-module for $\e=0$).

Then the Jacobi identity $\op{Jac}_\m(\La^3V)=0$ and the remark about invertibility imply that
$c=0$. Using the endomorphism freedom in the choice of basis in $V$, we can assume $b=\a\,i$.
Using again (\ref{Q}) we obtain $A_{i\ee}(x)=x(\frac\e2+ri)$ for some $r\in\R$.
Thus in this case $\m$ is a Lie algebra with the structure
equations\footnote{Here and in what follows we adopt the convention that $x,y\in V$ are arbitrary
elements, but $\ee\in\C$ is a fixed element, in particular $\ee,i\ee$ is a real basis of $\C$.}
 $$
\quad [x,y]=\a\Re(xi\bar{y})\ee,\ [x,i\ee]=x(\tfrac\e2+ri),\ [\ee,i\ee]=\e\ee.
 $$
In other words, $\m$ is obtained in two steps.
First we construct the central extension of abelian algebra $V$ by the 2-form $\oo_i$ -
we get the Heisenberg Lie algebra $\mathfrak{heis}(V)=V+\R\ee$. Then we take its
1-dimensional extension of by the derivation\footnote{The central ("left") extension
and extension by derivations ("right") $\tilde{\g}$ of the Lie algebra $\g$ (via $\mathfrak{f}$) are given
respectively by the exact sequences
 $$
0\to\mathfrak{f}\to\tilde{\g}\to\g\to0, \qquad
0\to\g\to\tilde{\g}\to\mathfrak{f}\to0.
 $$
Then $\g$ is respectively the quotient/subalgebra of $\tilde{\g}$ and its bracket can/cannot
change upon the extension.}
$\op{ad}_{i\ee}$. In Table A1 this case is called A1.1.

\medskip

(A1.2) $\l=0$, $[V,V]\ne0$ and the subalgebra $A_\C\subset\H$ is zero.
In this case the subalgebra $\C$ must be abelian, and then the Jacobi identity
holds $\op{Jac}_\m(\La^3V)=0$ identically.
We can normalize (by rescaling $\ee$) $b=i$ and $c=p\in\Im(\H)$ is arbitrary.

Thus $\m$ is the Lie algebra, which is a 2-dimensional central extension of abelian $V$ by two
2-forms $\oo_i,\oo_p$. This is the case A1.2 of Table A1.

\medskip

(A1.3) $\l=0$, $[V,V]=0$. Here the Jacobi identity is satisfied and we only need to normalize
the map $\eta\mapsto A_\eta$. If $\e\ne0$, then $A_\ee=0$ and we choose a basis in $V$ such that
$A_{i\ee}(x)=x(\b+ri)$, $\b,r\in\R$.

On the other hand, if $\e=0$, then $A_\C\subset\H$ is either 1-dimensional or 2-dimensional subalgebra,
which is possible only for $A_\ee(x)=x\a$, $A_{i\ee}(x)=x(\b+ri)$, $\a,\b,r\in\R$. Clearly we can normalize $\a=0$ or $1$.

Thus $\m$ is a Lie algebra, which is either an extension of abelian $V$ by two commuting derivations
$\op{ad}_\ee,\op{ad}_{i\ee}$ or an extension of abelian $V+\R\ee$ by one derivation $\op{ad}_{i\ee}$.
This is the case A1.3 of Table A1.

\medskip

(A1.4) $\l\neq0$. Here $\m$ is not a Lie subalgebra of $\g$.
In this case the Jacobi identity implies that $\C$ has a central element
(in particular $\C$ is abelian). We can choose it to be $i\ee$. Then $A_{i\ee}=0$, $c=0$.
We can normalize (by endomorphisms) $\l=\pm1$, $b=i$. The Jacobi identity further yields
$A_\ee(x)=3\l xi$.

The obtained Lie algebra $\g$ is isomorphic to $\mathfrak{u}_3$ for $\l=1$ and
to $\mathfrak{u}_{1,2}$ for $\l=-1$. The associated almost complex manifolds are
$U(3)/SU(2)$ and $U(2,1)/SU(2)$ respectively.
In Table A1 this case is called A1.4.

\medskip

Thus we obtained the complete classification of the homogeneous structures $M=G/H$
in the first case from the list of Theorem \ref{Thm3}. Almost complex structures on
$M$ are obtained from $\h$-invariant complex structures on $\m$ specified above.

The Nijenhuis tensor can be computed by the formula
 $$
N_J(X,Y)=\pi([JX,JY]-J[X,JY]-J[JX,Y]-[X,Y]),\quad X,Y\in\m,
 $$
where $[,]:\La^2\m\to\g$ are the brackets in the Lie algebra $\g$ and $\pi:\g\to\m$ is the projection
along $\h$. Similarly the differential of the almost symplectic structure $\oo$ on $\m$
is computed by the Cartan formula using only the brackets on $\m$. This explains all entries in
Table A1.

\medskip

Consider the other representation $\m=\mathfrak{ad}^\C$ of $\h=\mathfrak{su}(2)$.
Here it is more convenient to involve Levi decomposition: $\g=\g_{ss}\oplus\mathfrak{r}$,
where the first summand is a semi-simple part and the second is the radical.
The factor $\g_{ss}$ can be chosen to contain $\h$, and so can be one of the following:
 \begin{itemize}
\item $\g_{ss}=\mathfrak{su}(2)^3$, $\h=\mathfrak{su}(2)_\text{diag}$,
\item $\g_{ss}=\mathfrak{su}(2)\oplus\mathfrak{sl}(2,\C)$, $\h=\mathfrak{su}(2)_\text{diag}$,
\item $\g_{ss}=\mathfrak{su}(2)^2$, $\h=\mathfrak{su}(2)_\text{diag}$,
\item $\g_{ss}=\mathfrak{sl}(2,\C)$, $\h\subset\mathfrak{sl}(2,\C)$.
 \end{itemize}
The last case is disqualified as $\h$ acts by zero on $\g/\g_{ss}$
(no 3-dimensional nontrivial representation for $\g_{ss}$). The other give the cases A2.1, A2.2, A2.3
of Table A2 respectively.

Finally, it is possible that $\g_{ss}=\h$, whence $\m=\mathfrak{r}$. Since this
is solvable of module type $\mathfrak{ad}^\C$, the only non-trivial brackets
in terms of the splitting $\m=\m_1+\m_2$ (grading) are the following (case A2.4)
 $$
[x_1,y_1]=[x,y]_2,\quad x,y\in\mathfrak{ad}\simeq\h.
 $$

The corresponding analysis for $\h=\mathfrak{su}(1,1)$ is similar (for instance, Lemma \ref{Le2}
holds true with $\H$ changed to $\H_s$), but a special
care should be taken as in this non-compact case there are null elements on the representation
$V$ with respect to its unique (up to scale) $\h$-invariant metric (that's why Table A3 is bigger than A1).
The cases $\h=\mathfrak{sl}_2(\C),\mathfrak{sl}_3(\R),\mathfrak{sl}_3(\C)$ are much simpler.
The details of computations can be found in \cite{Wi}.
For the largest algebras $\mathfrak{su}(3)$ and $\mathfrak{su}(2,1)$
arising in the symmetry analysis the computations are done in \cite{K$_2$}.

This finishes the proof of Theorem \ref{Thm3}.

 \section{Automorphism groups of nondegenerate structures $J$}\label{S3}

Here we prove Theorem \ref{Thm2} -- find the symmetry algebra of the structures $J$
with non-degenerate Nijenhuis tensor $N_J$. It follows from the tables that $N_J$ can be non-degenerate only
for isotropy algebras $\mathfrak{su}(2)$, $\mathfrak{su}(1,1)$ or $\mathfrak{su}(3)$, $\mathfrak{su}(2,1)$.
In the latter two cases $J$ is locally isomorphic to the $G_2$-invariant almost complex structure on
either $S^6$ or $S^{2,4}$ and if the automorphism group has dimension 14, it is one of the two
forms of the group $G_2$, see \cite{K$_2$} for details.
In what follows we consider the former two cases.

According to \cite{K$_2$} the isotropy algebra $\mathfrak{sym}(J)_o$ of the symmetry algebra
$\mathfrak{sym}(J)$ at the point $o\in M$ is 1-jet determined.
Indeed, the proof of Theorem 1 loc.cit. implies
 \begin{theorem}\label{BK}
If the Nijenhuis tensor $N_J$ on a connected almost complex manifold $(M^6,J)$ is non-degenerate,
then any vector field $X\in\mathfrak{sym}(J)$, is uniquely determined by its 1-jet $[X]_o^1$.
Consequently, the isotropy algebra satisfies:
 $$
\mathfrak{sym}(J)_o=\{X:L_X(J)=0,X(o)=0\}\subset\mathfrak{gl}(\m,J).
 $$
 \end{theorem}

 \subsection{Symmetries via derivations}

Theorem \ref{BK} hints to the following statement concerning the symmetry algebra of the homogeneous models
A1.1, A2, A3.1, A3.2 and A4 according to the numeration in Appendix.

 \begin{prop}\label{Pro1}
Let $M=G/H$ be the homogeneous almost complex manifold
associated with one of the Lie algebra extensions of the $\h$-module $\m$
in the case when $\m$ is a Lie algebra (subalgebra in $\g$) and
$\h$ is either $\mathfrak{su}(2)$ or $\mathfrak{su}(1,1)$.
If the Nijenhuis tensor $N_J$ of the almost complex structure $J$ is
non-degenerate, then the full symmetry algebra as a vector space is
 $$
\mathfrak{sym}(J)=\m+\mathfrak{sym}(J)_o
 $$
and the full isotropy algebra equals
 $$
\mathfrak{sym}(J)_o=\mathfrak{der}(\m)\cap\mathfrak{gl}(\m,J)=\{A\in\mathfrak{der}(\m):AJ=JA\}.
 $$
 \end{prop}

Otherwise said, we are given the pair $(\g,\h)$ with $\m=\g/\h$ and $\h$-invariant $J$ on $\m$.
The claim is that if we can find an extension $(\tilde\g,\tilde\h)\supset(\g,\h)$ with the
same property, then still $\tilde\h$ acts on $\m$ by derivations.

Below we denote by $\pi:\tilde\g\to\m$ the projection along $\tilde\h$ and use the labels
for almost complex homogeneous spaces from the Appendix.

 \begin{proof}
Let $\tilde\g=\mathfrak{sym}(J)$ be the full symmetry algebra.
By \cite{K$_1$,K$_2$} the full isotropy algebra $\tilde\h=\mathfrak{sym}(J)_o\subset\tilde{\g}$
is at most 8-dimensional.
If $J$ is not locally isomorphic to the $G_2^c$-invariant almost complex structure on $\mathbb{S}^6$
or $G_2^*$-invariant almost complex structure on $S^{2,4}$, then $\dim\tilde\h\le5$.
Indeed, this is so in any non-exceptional case of canonical forms NDG(1-4) of
\cite{K$_1$}, exceptions are the two cases given by formulae (16) and (17) of \cite[Section 7]{K$_2$} when
$\tilde\h$ is equal to $\mathfrak{su}(2,1)$ or $\mathfrak{su}(3)$ respectively. The structure of $\g$
recovers uniquely and $M=G_2^*/SU(2,1)$ in the first case and $M=G_2^c/SU(3)$ in the second case respectively;
the structure $J$ in every case is unique. Any proper subalgebra of $\mathfrak{su}(2,1)$ is at most 5-dimensional,
and that of $\mathfrak{su}(3)$ is at most 4-dimensional, whence the claim.

Since already $J$ has 3-dimensional isotropy $\h$ by construction, the additional subspace
$\mathfrak{r}\subset\tilde{\h}$ has $\dim\mathfrak{r}\le2$.
Thus this $\mathfrak{r}$ can be chosen the radical of the full isotropy algebra $\tilde\h\subset\tilde\g$.

We start with the $\mathfrak{ad}^\C$ representation of $\h=\mathfrak{su}(2)$ on $\m$.
Then by $\h$-equivariance of the brackets, $\mathfrak{r}$ is in the radical
$\tilde{\mathfrak{r}}$ of $\tilde\g$
(this follows by dimensional reasons and Levi decomposition of $\mathfrak{r}\oplus\m'$, where
$\m'$ is the part of $\m$ that does not enter into the semi-simple part $\g_{ss}\subset\g$).

If $\m$ is either $\mathfrak{su}(2)\oplus\mathfrak{su}(2)$ or $\mathfrak{sl}(2,\C)$ (A2.1,A2.2),
then $\g$ is semi-simple, so extension of $\h$ to $\tilde\h$ is by radical $\mathfrak{r}$ only, whence
$[\mathfrak{r},\m]\subset\mathfrak{r}$ and hence the added summand $\mathfrak{r}$ acts non-effectively,
which is prohibited.

In the case $\m=\mathfrak{su}(2)+\R^3$ (A2.3) the radical of $\tilde\g$ is
$\tilde{\mathfrak{r}}=\mathfrak{r}+\R^3$
and so we have: $[\mathfrak{r},\mathfrak{su}(2)]=0$
 % first $\subset\mathfrak{r}$ and then =0 by $\h$-equivariance
and $[\mathfrak{r},\R^3]\subset\mathfrak{r}+\R^3$. By $J$-invariancy of $\pi\circ\mathfrak{ad}_\mathfrak{r}(\m)$ we get $\pi([\mathfrak{r},\m])=0$.
This means that the action of $\mathfrak{r}$ on $\m$ is non-effective and we exclude such
$\mathfrak{r}$ as before.

In the last $\mathfrak{ad}^\C$-case $\m=\mathfrak{a}_1\oplus\mathfrak{a}_2$ (A2.4)
the Lie algebra structure is graded and the $\h$-representation
on $\mathfrak{a}_i\simeq\R^3$ is adjoint. In the radical $\tilde{\mathfrak{r}}=\m+\mathfrak{r}$
the $\h$-representation $\mathfrak{r}$ is the trivial submodule, whence by $\h$-equivariance
$[\mathfrak{r},\m]\subset\m$. This means that $\mathfrak{r}$ acts by derivations, as required
in Proposition \ref{Pro1}.

The case $\h=\mathfrak{su}(1,1)$ is similar except for the last type $\m=\mathfrak{a}_1\oplus\mathfrak{a}_2$. Then another possibility occurs that $\mathfrak{r}$ is
the standard representation $\mathcal{R}(\l_1)$ of $\h\simeq\mathfrak{sl}_2(\R)$,
where $\mathcal{R}(w)$ is the representation
of highest weight $w$. As $\mathfrak{a}_i=\mathcal{R}(2\l_1)$ and
$\mathcal{R}(\l_1)\otimes \mathcal{R}(2\l_1)=\mathcal{R}(\l_1)+\mathcal{R}(3\l_1)$,
$\h$-equivariance implies $[\mathfrak{r},\m]\subset\mathfrak{r}$, so the action is non-effective.

Consider now the second possible representation of $\h$: $\m=V+\C$. Starting with $\h=\mathfrak{su}(2)$
we note the additional subspace $\mathfrak{r}$ is the trivial $\h$-module.
Again $\h$-equivariance (together with Schur's lemma) implies $[\mathfrak{r},V]\subset V$.
Since $0\neq[V,V]\subset\C$ and $\pi\circ\mathfrak{ad}_\mathfrak{r}|_{\m}$ commutes with $J$,
the Jacobi identity implies $[\mathfrak{r},\C]\subset\langle[V,V],J[V,V]\rangle=\C$. Thus $[\mathfrak{r},\m]\subset\m$
and $\mathfrak{r}$ acts by derivations, as claimed in Proposition \ref{Pro1}.

When $\h=\mathfrak{su}(1,1)$ and $\mathfrak{r}$ as $\h$-representation is trivial,
the argument is the same (here $\dim\mathfrak{r}\le2$).
Consider the case $\mathfrak{r}\simeq\R^2$ being the standard representation $U$ of
$\h\simeq\mathfrak{sl}_2(\R)$. In this case $V=U^\C$ is $2\mathcal{R}(\l_1)$ as $\h$-representation.
Therefore as $\mathcal{R}(\l_1)\otimes \mathcal{R}(\l_1)=\mathcal{R}(0)+\mathcal{R}(2\l_1)$
(where the first summand is the trivial 1-dimensional representation and the second is the adjoint),
we get by $\h$-equivariance
$[\mathfrak{r},V]\subset\h+\C$. Since $0\neq[V,V]\subset\C$, $[V,\C]\subset V$, $[\h,V]=V$
and $\pi\circ\mathfrak{ad}_\mathfrak{r}|_{\m}$ commutes with $J$, the Jacobi identity
yields $[\mathfrak{r},\C]\subset[\h+\C,V]=V$.

Let us show that $[\mathfrak{r},V]$ does not have $\h$ component. Indeed, since
$[\mathfrak{r},\mathfrak{r}]\subset\h+\C$ by $\h$-equivariance, the Jacobi identity implies
 $$
V\supset[[\mathfrak{r},\mathfrak{r}],V]\subset[\mathfrak{r},[\mathfrak{r},V]]\subset
[\mathfrak{r},\h+\C]\subset\mathfrak{r}+V.
 $$
Since $[\mathfrak{r},\h]=\mathfrak{r}$, the presence of $\h$ component implies non-trivial
$\mathfrak{r}$ component in the last summand of the above display formula.
To study it consider the bracket-maps $\psi:\mathfrak{r}\times\h\to\mathfrak{r}$ and
$\phi:\mathfrak{r}\times V\to\h$ (for the latter we post-compose with the projection).
The Jacobi identity and the above display formula imply $\psi(r_1,\phi(r_2,v))=\psi(r_2,\phi(r_1,v))$
for all $r_1,r_2\in\mathfrak{r}$, $v\in V$. Since the maps $\psi,\phi$ depend only on the $\h$-module
structure and as such $\mathfrak{r}=U$, $V=U\oplus U$, we change the maps to
 $$
\Psi:U\times\h\to U,\ \Phi:U\times U\to\h \text{ with }
\Psi(u_1,\Phi(u_2,u_3))=\Psi(u_2,\Phi(u_1,u_3)).
 $$
Here $\Psi$ is the standard representation, and $\Phi$ is proportional to the symmetric
multiplication $(u_1,u_2)\mapsto\l\,u_1u_2$, $\l\in\R$, because $S^2U=\h$.
This isomorphism is given by a choice of $\h$-invariant area form $\oo$
on $U$. Let $p,q\in U$ be the canonical basis, $\oo(p,q)=1$. Then
 $$
\Psi(p,\Phi(q,p))=\l\Psi(p,qp)=\l p\ne\Psi(q,\Phi(p,p))=\l\Psi(q,p^2)=-2\l p
 $$
unless $\l=0$. Thus the $\h$ component vanishes and $[\mathfrak{r},V]\subset\C$.

Therefore $[\mathfrak{r},\m]\subset\m$ and
$\mathfrak{r}$ acts by derivations.

This finishes the proof of Proposition \ref{Pro1}.
 \end{proof}

 \subsection{Proof of Theorem \ref{Thm2}}

To find the derivations we can use the exact sequence
 \begin{equation}\label{loseq}
0\to Z(\m)\longrightarrow\m\stackrel{\op{ad}}\longrightarrow\mathfrak{der}(\m)\longrightarrow H^1(\m,\m)\to0,
 \end{equation}
where $Z(\m)$ is the center of the Lie algebra $\m$.

Consider at first the case $\h=\mathfrak{su}(2)$, representation $\mathfrak{ad}^\C$.
In all four cases here $\m\subset\g$ is a Lie subalgebra.

In the first two cases A2.1, A2.2 it is semi-simple:
$\mathfrak{su}(2)\oplus\mathfrak{su}(2)$ or $\mathfrak{sl}_2(\C)$.
By Whitehead lemma $H^1(\m,\m)=0$, so all derivations are inner. Thus
$\mathfrak{der}(\m)=\{\op{ad}_X:X\in\m\}\simeq\m$ from (\ref{loseq}).

We claim that if $\op{ad}_X$ commutes with $J$, then $\op{ad}_{JX}$ does not. Elsewise
 $$
N_J(X,Y)=[\op{ad}_{JX},J](Y)-J[\op{ad}_X,J](Y)=0,
 $$
and the Nijenhuis tensor is degenerate (even DG$_2$ in terminology of \cite{K$_1$}).
Therefore $\mathfrak{sym}(J)_o$ is totally real in $\m$ and so cannot have dimension $>3$.
But dimension 3 is guaranteed since $\mathfrak{sym}(J)_o\supset\h$. Consequently
$\mathfrak{sym}(J)_o=\h$ and $\mathfrak{sym}(J)=\g$.

Next case A2.3 is $\m=\h\ltimes\R^3$. Clearly
$\mathfrak{sym}(J)_o\subset\mathfrak{der}(\m)$ must preserve the radical $\R^3$.
The operator $J_r=J-r\1$ is invariant, where $\1$ is the identity operator -- see formula (\ref{J}).
Therefore the full isotropy also preserves the semi-simple part $J_r\R^3=\h$,
and the action on $\h$ induces the action on $\R^3$.
Again any derivation on $\h$ is internal by the Whitehead lemma and so
$\dim\mathfrak{sym}(J)_o\le3$ implying $\mathfrak{sym}(J)=\g$.

The last case A2.4 is the graded nilpotent Lie algebra $\m=\mathfrak{a}_1\oplus\mathfrak{a}_2$
with $\mathfrak{a}_i\simeq\R^3$, $J_r:\mathfrak{a}_1\to\mathfrak{a}_2$ and the bracket
being given by $[\xi,\eta]=J_r(\xi\times\eta)$ for $\xi,\eta\in\mathfrak{a}_1$,
the cross product $\times$ being the Lie bracket on
$\R^3=\mathfrak{su}(2)$. This relation shows that $\mathfrak{a}_2=[\m,\m]$ equipped with $\times$
product (so the bracket is
$\La^2\mathfrak{a}_2\ni\xi\we\eta\mapsto[J_r^{-1}\xi,J_r^{-1}\eta]\in\mathfrak{a}_2$)
must be preserved by the derivations. Since this algebra $(\mathfrak{a}_2,\times)$ is isomorphic
to $\mathfrak{su}(2)$, we obtain $\dim\mathfrak{sym}(J)_o\le3$ and $\mathfrak{sym}(J)=\g$ as before.

Finally there is a family of representations of $\h$ of the type $V\oplus\C$.
Only one of the cases A1.1, with $\m$ being a Lie algebra, has non-degenerate $N_J$.
In this case the space of derivations of $\m$ commuting with $J$ is obtained by
the straightforward computation with the case split according to
parameters; these tedious computations are done in Maple. The result is the same as above.

The case $\h=\mathfrak{su}(1,1)$ is very similar to the considered $\mathfrak{su}(2)$.
The only difference is that in $V+\C$ representation there is one more case.

\smallskip

Now to complete the proof of Theorem \ref{Thm2}, we have to consider the homogeneous structures
of type II in Theorem \ref{Thm1} (when $\m$ is not a Lie algebra: A1.4, A3.5).
For such $M=G/H$ the Lie group $G$ is reductive, $\g=\g_{ss}+\mathfrak{z}$
with 1-dimensional center $\mathfrak{z}$ and 8-dimensional
semi-simple part $\g_{ss}$, and the Lie algebra $\tilde\g=\mathfrak{sym}(J)$ contains $\g$.

From the proof of Proposition \ref{Pro1} we know that
$\tilde\h=\h+\mathfrak{r}$, where the semi-direct summand
$\mathfrak{r}$ is the radical in $\tilde\h$ and
$\dim\mathfrak{r}\le2$.
This implies (by inspection of the Levi decomposition of $\tilde\g$) that
$\mathfrak{z}+\mathfrak{r}\subset\tilde\g$ is a subalgebra, which is either
semi-simple or the radical of $\tilde\g$. In any case, because
$\m=\m_0+\mathfrak{z}$ for $\m_0=\m\cap\g_{ss}\subset\g_{ss}$, we get
$[\mathfrak{r},\m]\subset[\mathfrak{r},\m_0]+[\mathfrak{r},\mathfrak{z}]\subset\mathfrak{r}+\mathfrak{z}$.
Consequently $\pi\circ\mathfrak{ad}_\mathfrak{r}(\m)\subset\mathfrak{z}$ for the projection $\pi:\tilde\g\to\m$ along $\tilde\h$. Since $\pi\circ\mathfrak{ad}_\mathfrak{r}\subset\mathfrak{gl}(\m,J)$ yields $J$-invariance of $\pi\circ\mathfrak{ad}_\mathfrak{r}(\m)$, we conclude that by dimensional reasons $\pi([\mathfrak{r},\m])=0$.
Consequently the action of $\mathfrak{r}$ on $\m$ is not effective, so $\mathfrak{r}=0$.

This finishes the proof of Theorem \ref{Thm2}.

 \section{Almost Hermitian structures and their integrability}\label{S4}

The existence of a homogeneous almost Hermitian, almost symplectic or almost complex structure
depends only on the isotropy representation, in contrast with the various integrability conditions
(K\"ahler, etc.) for such structures which generally depend on the Lie algebra structure.

Pseudo-Riemannian metrics on the almost complex homogeneous manifold $M=G/H$
with the isotropy $\h$-module $\m$ and $\h$-invariant complex structure $J$ on $\m$ correspond
to non-degenerate $\h$-invariant quadratic forms $g\in S^2\m^*$. Invariant almost Hermitian structures are
elements of the set
 $$
S^2_J\m^*=\{g\in(S^2\m^*)^\h:g(J\xi,J\eta)=g(\xi,\eta),\det(g)\neq0\}.
 $$

Likewise invariant compatible almost symplectic structures are elements of the set
 $$
\La^2_J\m^*=\{\oo\in(\La^2\m^*)^\h:\oo(J\xi,J\eta)=\oo(\xi,\eta),\oo^3\neq0\}.
 $$
The K\"ahler form $\oo\in\La^2_J\m^*$ associated to $g\in S^2_J\m^*$ is defined by
$\oo(\xi,\eta)=g(J\xi,\eta)$. This formula makes a bijective correspondence $S^2_J\m^*\simeq\La^2_J\m^*$.

Note that two invariant almost Hermitian metrics $g,\tilde g$ define a symmetric (with respect to both $g$
and $\tilde g$) invertible operator $A:\m\to\m$ by $\tilde g(\xi,\eta)=g(A\xi,\eta)$.
Since both sides are $\mathfrak{h}$-invariant, this $A$ commutes with $\mathfrak{h}$.
Moreover the compatibility condition implies that $A$ is complex linear, $[A,J]=0$.
Thus the operator $A$ belongs to the complex endomorphism ring $\op{End}_\h(\m,J)$.

 \subsection{Classification of almost Hermitian structures}

Let us list all invariant almost Hermitian structures according to the types of $\h$-modules as in
Theorem \ref{Thm3} (we'll omit the word "almost" for the metric).

 \smallskip

{\bf Case 1:} $\h=\mathfrak{su}(2)$, $\m=V\oplus\C$, where $V\simeq\H$.
There are Hermitian metrics $g_1$ on $V$, $g_2$ on $\mathbb{C}$.
Since $\op{End}_\h(\m,J)=\C\oplus\C$ (acting submodule-wise)
and the only symmetric endomorphism are $A\in\R\oplus\R$,
the general invariant compatible metric on $\mathfrak{m}$ is $g=a\,g_1+b\,g_2$.
The signature of $g$ is any even $(2k,6-2k)$ depending on $a,b\neq0$.

The almost-symplectic form on $\C$ component is unique up to scaling (which we fix).
Since the endomorphism ring of the first component is $\op{End}_\h(V)=\H$
and only imaginary quaternions are skew-symmetric, we conclude
$\omega(\xi,\eta)=g_1(\xi_1q,\eta_1)+g_2(i\xi_2,\eta_2)$ for some $q\in\Im(\H)\setminus0$,
where $\xi=\xi_1+\xi_2,\eta=\eta_1+\eta_2\in V\oplus\C$.

\smallskip

{\bf Case 2:} $\h=\mathfrak{su}(2)$, $\m=\mathfrak{ad}^\C$.
The operator $J$ induces the equivariant splitting $\m=\mathfrak{ad}\oplus J\mathfrak{ad}$.
The Riemannian metric  $g$, which is the direct sum of the Killing forms on each summand
is compatible. Since $\m$ is a complex irreducible representation, $\op{End}_\h(\m,J)=\C$.
Then $g(A\xi,\eta)$ for $A\in\op{End}_\h(\m,J)$ is non-symmetric unless $A\in\R$,
and so the invariant compatible metric is unique up to scaling (but $g$ depends on $J$).

Decomposing $\h$-modules $\La^2\m=3\La^2\mathfrak{ad}\oplus S^2\mathfrak{ad}=
3\mathfrak{ad}\oplus W^5\oplus\R^1$, where $W^5$ is irreducible and $\R^1$ trivial, we
conclude that the only almost symplectic form comes from the compatible metric.

\smallskip

{\bf Case 3:} $\h=\mathfrak{su}(1,1)$, $\m=V\oplus\C$.
Here $V=U^\mathbb{C}$ for the standard $\mathfrak{sl}_2$ representation $U$.
This case is similar to Case 1, and the general invariant metric is again
$g=a\,g_1+b\,g_2$, $a,b\neq0$. Now however the metric $g_1$ is of split
signature, $g_2$ is Riemannian, and so $g$ has type $(4,2)$ (or $(2,4)$, but we will
not distinguish the opposite signatures).

Since $\La^2V=\mathfrak{ad}\oplus\R^3$, the space of invariant 2-forms is 3-dimensional.
Indeed, $\op{End}_\h(V)=\mathfrak{gl}_2$, and so a general almost symplectic form is
$\omega(\xi,\eta)=g_1(A\xi_1,\eta_1)+g_2(i\xi,\eta)$, $A\in\op{GL}_2$.
We can write this via split quaternions $\H_s$ (isomorphic to $\mathfrak{gl}_2$ as algebra) so:
$\omega(\xi,\eta)=g_1(\xi_1q,\eta_1)+g_2(i\xi_2,\eta_2)$, $q\in\Im(\H_s)$, $q^2\neq0$.

\smallskip

{\bf Case 4:} $\h=\mathfrak{su}(1,1)$, $\m=\mathfrak{ad}^\C$.
This is similar to Case 2: the (invariant) almost symplectic structure $\oo$ is unique up to scale;
it is $J$-independent and is $J$-compatible for every $J\in\op{End}_\h(\m)$, $J^2=-\1$. The Hermitian metric
$g=-i_J\oo$ depends on $J$ and has signature $(4,2)$.

\smallskip

{\bf Case 5:} $\h=\mathfrak{sl}_2(\C)$, $\m=V+\C$.
Since $\mathfrak{sl}_2(\C)$ contains the subalgebras $\mathfrak{su}(2)$ and $\mathfrak{su}(1,1)$, the
metric on $V$ component must be invariant under both of them, and no such metric exists.

Consider the 2-form $\oo_1(x,y)=g_1(xq,y)$ on $V$ invariant with respect to $\mathfrak{su}(2)$, $q\in\Im(\H)$, see Case 1. Since $\mathfrak{su}(1,1)$ can be identified with $i\mathfrak{su}(2)$
inside $\mathfrak{sl}_2(\C)$, we compute that $\oo_1$ is $\mathfrak{su}(1,1)$-invariant iff
$q\in\Im(\H)\cap i\Im(\H)$, i.e. $q\perp i$. On the $\C$ factor the invariant symplectic
form is unique. Thus the space of invariant almost symplectic structures on $\m$ is given by 2
parameters up to scale: $\oo=\oo_1+\oo_2$.

\smallskip

{\bf Case 6:} $\h=\mathfrak{sl}_2(\C)$, $\m=\mathfrak{ad}$.
The Killing form $K$ provides an invariant metric on $\mathfrak{ad}$, but it is not Hermitian as
$K(J\xi,J\eta)=-K(\xi,\eta)$. Since any other metric or 2-form must be related to $K$ by an
operator $A\in\op{End}_\h(\m)=\C$, no compatible metric and no almost-symplectic form exists.
Instead we have two invariant anti-compatible metrics $K(\xi,\eta)$ and $K(J\xi,\eta)$.
The almost complex structure $J=i$ on $\m$, and hence on the corresponding homogeneous manifold
$M=\op{SL}_2(\C)\times\op{SL}_2(\C)/\op{SL}_2(\C)_\text{diag}\simeq\op{SL}_2(\C)$, the structure $J$ is integrable.

\smallskip

{\bf Case 7:} $\h=\mathfrak{su}(3)$, $\m=V$.
By definition we have an invariant Hermitian metric $g$ of signature $(6,0)$ on $\m$.
The endomorphism ring is $\C$, but $g(A\xi,\eta)$ is not symmetric unless $A\in\R$, so
$g$ is unique up to scaling. The almost symplectic form is also unique and compatible.

The corresponding homogeneous manifold is $\mathbb{S}^6$, and $J$ is the unique invariant
almost complex structure. Known as the Calabi structure, it is well studied. In particular,
the triple $(g,J,\oo)$ is strongly nearly K\"ahler (SNK) and the Hermitian metric $g$ is
3-symmetric and Einstein.

\smallskip

{\bf Case 8:} $\h=\mathfrak{su}(2,1)$, $\m=V$.
By definition we have an invariant pseudo-Hermitian metric $g$ of signature $(4,2)$ on $\m$.
The endomorphism ring is $\C$, but $g(A\xi,\eta)$ is not symmetric unless $A\in\R$, so
again the pseudo-Riemannian metric $g$ and the almost symplectic form $\oo$ are both unique
(up to scaling) and compatible.

The corresponding homogeneous manifold is $\mathbb{S}^{2,4}\simeq \mathbb{S}^2\times\R^4$, and $J$ is
the unique invariant almost complex structure. It is not integrable, has non-degenerate Nijenhuis tensor,
and is the split analog of the Calabi structure. The triple $(g,J,\oo)$ is strongly nearly
pseudo-K\"ahler and the Hermitian metric $g$ is Einstein.

 \subsection{K\"ahler and nearly K\"ahler structures}

Examining the list of all our homogeneous structures we conclude that the only K\"ahler metrics
are the cases A1.1, A3.1 and A1.3, A3.4 of Tables from the Appendix (we pair the similar cases).
Even though the groups on which the structures live are solvable (the
topology is rather simple), the metric properties are non-trivial. We summarize the results.

 \begin{theorem}
The only (pseudo-)K\"ahler homogeneous 6D manifolds with semi-simple (nontrivial) isotropy are
quotients $M=G/H$ with $H=SU(2)$ or $H=SU(1,1)$ with reducible isotropy representation $\m=V+\C$.
As an $H$-module $V=\H$ or resp. $V=\H_s$.

The corresponding reductive complement $\m$ is a Lie algebra, so $M$ is a (quotient of) Lie group
with the Lie algebra given by the following relations (two cases). Below $\a,r\in\R$, $\varepsilon\in\{0,1\}$ are
the parameters, $\a\neq0$, and the
vectors $x,y\in V,\ \ee,i\ee\in\C$.
 \smallskip
 $$
[x,y]=\a\,\Re(x i\bar y)\ee,\ [x,i\ee]=x(\tfrac12+ri),\ [\ee,i\ee]=\ee.
\leqno{1)}\quad\m:
 $$
Thus $\m^6$ is 1-dimensional "right\footnote{This means extension by derivations; terminology
comes from Fuks \cite{F}, and is opposed to left=central extensions. For a Lie algebra $\g$
its "right extensions" are enumerated by the cohomology group $H^1(\g,\g)$ and "left extensions"
by $H^2(\g)$.}\ extension" of the 5D
Heisenberg algebra. The symplectic form is $\oo=\a\,\oo_V+\oo_\C$,
$\oo_V(x,y)=\Re(x i\bar y)$, $\oo_\C(\ee,i\ee)=1$. The Hermitian metric is $g(\xi,\eta)=\oo(\xi,J\eta)$
with $J(x,\ee)=(xi,i\ee)$ in the decomposition $V+\C$, with signature $(6,0)$ for $\h=\mathfrak{su}(2)$
and $\a>0$ and signature $(4,2)$ else.
 % ($\a<0$ or $\h=\mathfrak{su}(1,1)$).
Moreover $g$ is Einstein with the cosmological constant $-4$, and is not conformally flat.
\smallskip
 $$
[x,i\ee]=rxi,\ [\ee,i\ee]=\varepsilon\ee.
\leqno{2)}\quad\m:
 $$
Thus $\m^6$ is a 1-dimensional "right extension" of the 5D
Abelian algebra. The symplectic form is $\oo=\oo_V+c\,\oo_\C$.
The Hermitian metric is again $g(\xi,\eta)=\oo(\xi,J\eta)$ with $J(x,\ee)=(xi,i\ee)$.
It has signature $(6,0)$ for $\h=\mathfrak{su}(2)$, $c>0$ and signature $(4,2)$ else.
The metric is not Einstein or conformally flat unless $\varepsilon=0$, when $g$ is flat.
 \end{theorem}

It is also interesting to study when $(M,g,J,\oo)$ is strongly nearly K\"ahler (SNK), meaning
that for the Levi-Civita connection $\nabla$ the tensor $\nabla\oo$ is (nonzero) totally skew symmetric
($\nabla\oo=\frac13d\oo\ne0$). It is a restrictive condition. For instance, the Nijenhuis tensor $N_J$ is non-degenerate
and the geometry is constrained by the 'splitting principle' of P.-A. Nagy \cite{Na}.
Homogeneous SNK structures were classified by J.-B. Butruilles \cite{Bu}.
The classification up to a covering is:
 \begin{itemize}
\item $\mathbb{S}^6=G_2^c/SU(3)$,
\item $\mathbb{S}^3\times \mathbb{S}^3=SU(2)\times SU(2)\times SU(2)/SU(2)_\text{diag}$
\item $\C P^3=SU(4)/(SU(3)\times U(1))=Sp(4)/(SU(2)\times U(1))$,
\item the flag variety $\mathbb{F}(1,2)=SU(3)/(U(1)\times U(1))$.
 \end{itemize}
The first two belong to our list
(the invariant structure $J$ on $\mathbb{S}^3\times \mathbb{S}^3$ corresponds to
the case A2.1 from the Tables with parameters $(r,t)=\pm\bigl(\frac1{\sqrt{3}},\frac2{\sqrt{3}}\bigr)$,
so it has more symmetry than observed in \cite{Bu}), while the
last two do not (as they have reductive and Abelian isotropy respectively, and can be detected by
a refinement of our calculation).

Homogeneous pseudo-SNK of signature $(2,4)$ (this is given by the same condition:
$\nabla\oo$ nonzero totally skew symmetric) with semi-simple isotropy can be extracted
from our classification\footnote{There are obvious pseudo-SNK analogs of signature $(2,4)$ of the last
two entries in Bitruilles' list, but we present here only the spaces $G/H$ with semi-simple $H$.}:
 \begin{itemize}
\item $\mathbb{S}^{2,4}=G_2^*/SU(1,2)$,
\item
$SL(2)\times SL(2)=SU(2,1)\times SU(2,1)\times SU(2,1)/SU(2,1)_\text{diag}$
(the invariant structure $J$ on this $M^6$ corresponds to
the case A4.1 of the Tables with parameters $(r,t)=\pm\bigl(\frac1{\sqrt{3}},\frac2{\sqrt{3}}\bigr)$)
\item the left-invariant structure on the (solvable) Lie group with Lie algebra from the
case A3.2 of the Tables with parameters (after rescaling $\oo$) equal to
$r=-\frac3{2t}$, $\epsilon=+1$, $\a=0$, $p=t(i+j)\in\mathbb{H}_s$, $u=-\frac12k\in\H_s$,
$q=i$ and $b=\frac12t\,i\in\mathbb{H}_s$.
 \end{itemize}

 \begin{rk}
Since the latter homogeneous pseudo-SNK structure does not have an SNK analog (the first two have obviously
the same complexification as the first two entries of Butruilles' list), let us write the structures
explicitly. The structure equations of the Lie algebra $\m$ are ($t\neq0$):
 \begin{gather*}
[e_1,e_2]=te_5, [e_1,e_3]=-te_5, [e_2,e_4]=-te_5, [e_3,e_4]=te_5, \\
[e_1, e_5]=-\tfrac32(e_2+e_3), [e_2,e_5]=\tfrac32(e_1-e_4), [e_3,e_5]=\tfrac32(e_4-e_1), \\
[e_4,e_5]=-\tfrac32(e_2+e_3), \
[e_1,e_6]=e_1-\tfrac12e_4, [e_2,e_6]=e_2+\tfrac12e_3, \\
[e_3,e_6]=e_3+\tfrac12e_2, [e_4,e_6]=e_4-\tfrac12e_1, [e_5,e_6]=e_5
 \end{gather*}
and the almost complex structure $J$ and the metric $g$ are given by the formulae
(we denote by $\te_i$ the dual basis of $\m^*$)
 \begin{gather*}
J=(e_2\ot\te_1-e_1\ot\te_2+e_3\ot\te_4-e_4\ot\te_3)+(e_6\ot\te_5-e_5\ot\te_6),\\
g=\tfrac12\,t\,(\te_1^2+\te_2^2-\te_3^2-\te_4^2)+\te_5^2+\te_6^2 \
 \end{gather*}
Notice that though this (1-parametric) structure lives on a Lie group $M^6$ (or its finite quotient),
it has the symmetry group of dimension 9 and
so can be represented on the homogeneous space $G^9/H^3\simeq M^6$.
 \end{rk}

 \subsection{SKT and Gray-Hervella classes}

In the next sections we will also study the strong K\"ahler with torsion (SKT) property
$\partial\bar\partial\oo=0$ (in addition to $N_J=0$) important in generalized K\"ahler geometry and supersymmetric
nonlinear sigma models. The named property is equivalent to
 $$
d_J^2\oo=0,\qquad d_J=d\circ J
 $$
(where $J\circ\z=\z(J\cdot,J\cdot,..)$), and we shall study generalizations when
$d_J^k\oo=0$ for larger $k$ (and $J$ not necessarily integrable).
For instance, the standard almost Hermitian structure $(g,J,\oo)$
on $\mathbb{S}^6$ is not SKT and $d_J^3\oo\neq0$, but the derived top form vanishes: $d_J^4\oo=0$.

There are many structures $J$ of type III, which are not SKT, but satisfy the condition $d_J^3\oo=0$.
The only occasions of SKT are these:

 \begin{theorem}
The only homogeneous Hermitian manifolds $M^6$ with semi-simple isotropy, which
satisfy the SKT property but do not belong to either K\"ah\-ler or pseudo-K\"ahler class,
are equivalent to the following.

\smallskip

1) The structure of case A1.2 with parameters $q=\cos\theta\cdot i+\sin\theta\cdot j$,
$p=\pm\sqrt{3\sin^2\theta-1}\cdot q+\sin\theta\cdot k$.
The Lie algebra $\m$ is the central extension
 $$
0\to\R^2\longrightarrow\m\longrightarrow\R^4\to0,
 $$
whence the homogeneous space is $M=G/H$ is an $\R^2$-bundle over $\R^4$.

\smallskip

2) The structure of case A1.3 with parameters $q=i$, $\a=0$, $\beta=-\frac12$, $\epsilon=1$
or of case A3.4 with the same parameters and in addition $p=0$, $u=\lambda\,i$.
The Lie algebra $\m$ is the "right" extension
 $$
0\to\R^4\longrightarrow\m\longrightarrow\mathfrak{s}_2\to0,
 $$
where $\mathfrak{s}_2=\op{Lie}(S_2)$ is the solvable non-abelian 2D Lie algebra
of the Lie group $S_2$, represented via rank 1 homomorphism $S_2\to\C^*\stackrel{\op{diag}}\hookrightarrow\op{GL}_2(\C)\subset\op{GL}_4(\R)$
(in both cases $\R^4=V=\H$ or $\H_s$ is equipped with the complex structure $i$).
The homogeneous space is $M=G/H\simeq\R^4\rtimes S_2$.

\smallskip

3) The structure of case A3.4 with parameters $q=i$, $\a=0$, $\epsilon=1$, $p=\frac12(i+j)$, $u=-\frac12k$ and
$\beta=-1$ or $\beta=\frac12$. The Lie algebra $\m$ is the "right" extension given by the same sequence as in 2),
but now the homomorphism $\vp:S_2\to\op{GL}_2(\C)\subset\op{GL}_4(\R)$ has rank 2:
$\vp(\ee)=R_{\frac12(i+j)}$, $\vp(i\ee)=R_{\beta-\frac12k}$, where $R_h$ is the right multiplication by
the split-quaternion $h$. Again the homogeneous space is $M=G/H\simeq\R^4\rtimes S_2$.
 \end{theorem}

It is also interesting which Gray-Hervella (GH) classes of almost Hermitian manifolds
are realizable within our homogeneous 6D manifolds with semi-simple isotropy
(in our classification only Tables A1 and A2 correspond to almost Hermitian manifolds,
so in what follows we refer only to them).
In the work \cite{GH} 16 classes of such manifolds were encoded by the
set $\mathcal{P}(\{1,2,3,4\})$ of all subsets of the 4-point set.

The class corresponding to the empty set $\emptyset$ is the K\"ahler class $\mathcal{K}$.
The class $\mathcal{W}$ corresponding to the whole set $\{1,2,3,4\}$ consists of all almost Hermitian manifolds.
The basic GH-classes $\mathcal{W}_i$, $1\le i\le4$, correspond to 1-point subsets.
The other classes correspond to the direct sum of representations $W_i$ of the basic
classes (equivalently the basic modules can be taken $W\ominus W_i$, the other modules
being the intersections of these; the non-trivial GH-classes are then given by the union of the conditions determining
$\sum_{j\neq i}\mathcal{W}_j$).

In order to represent the general class $\mathcal{W}$ by a disjoint union,
we modify the Gray-Hervella classes to the classes $\tilde{\mathcal{W}}_\z$, $\z\in\mathcal{P}(\{1,2,3,4\})$.
Namely, $\tilde{\mathcal{W}}_\z$ consists of all elements of $\sum_{i\in\z}\mathcal{W}_i$,
which do not belong to $\tilde{\mathcal{W}}_\t$ with $\t\varsubsetneq\z$.
This definition is inductive starting from $\tilde{\mathcal{W}}_\emptyset=\mathcal{K}$.
For instance, $\tilde{\mathcal{W}}_i=\mathcal{W}_i\setminus\mathcal{K}$,
$\tilde{\mathcal{W}}_{i,j}=(\mathcal{W}_i\oplus\mathcal{W}_j)\setminus(\mathcal{W}_i\cup\mathcal{W}_j)$
for $i\neq j$ etc.

We already discussed and classified the K\"ahler $\mathcal{K}$ and nearly K\"ahler $\mathcal{W}_1=\mathcal{NK}$ classes,
which also gives the description of strictly nearly K\"ahler class $\tilde{\mathcal{W}}_1$.
 % Let us exclude the K\"ahler structures below (since obviously $\mathcal{K}\subset\mathcal{W}_i$).
Inspection of the tables rules out the class $\mathcal{W}_2=\mathcal{AK}$ of almost K\"ahler manifolds
with the exception of the K\"ahler structures
(observe from the Tables A1-A2 that whenever $d\oo=0$, then also $N_J=0$,
so the structure belongs to the K\"ahler class $\mathcal{K}$).
The class $\mathcal{W}_3=\mathcal{H}\cap\mathcal{SK}$ of special Hermitian manifolds
is realized only (again with the exception of the K\"ahler class $\mathcal{K}$) by the Tables A1.2
(several parametric cases) and A2.2 ($r=0,t=\pm1$); the latter is
$M=SL_2(\C)$ with the standard complex structure.
The class $\tilde{\mathcal{W}}_4$ (containing locally conformally K\"ahler but nohn-K\"ahler manifolds)
is realized only by the Table A1.4 ($q=\pm i$), in particular for $\e=-1$
we obtain the Calabi-Eckmann structure on $\mathbb{S}^1\times\mathbb{S}^5$.
For the other classes we have:

 \begin{theorem}
The GH-classes, realized as homogeneous almost Hermitian manifolds $M^6$ with semi-simple isotropy,
are precisely the following:
$\mathcal{K}$, $\tilde{\mathcal{W}}_1$, $\tilde{\mathcal{W}}_3$, $\tilde{\mathcal{W}}_4$,
$\tilde{\mathcal{W}}_{1,2}$, $\tilde{\mathcal{W}}_{1,3}$, $\tilde{\mathcal{W}}_{2,3}$,
$\tilde{\mathcal{W}}_{3,4}$, $\tilde{\mathcal{W}}_{1,2,3}$, $\tilde{\mathcal{W}}_{1,2,4}$, 
$\tilde{\mathcal{W}}_{1,3,4}$, $\tilde{\mathcal{W}}_{2,3,4}$,
$\tilde{\mathcal{W}}=\tilde{\mathcal{W}}_{1,2,3,4}$.
 \end{theorem}

Thus the non-realizable (via our models) GH-classes are 
$\tilde{\mathcal{W}}_2$, $\tilde{\mathcal{W}}_{1,4}$, $\tilde{\mathcal{W}}_{2,4}$.

 \begin{proof}
The proof is the direct calculation (in this algebraic computation we used Maple).
Let us indicate, which sub-classes in Tables A1 and A2 realize the GH-classes
(omitting the precise values of the parameters):

 \begin{tabbing}
\quad\ \= A1.4: $\tilde{\mathcal{W}}_4$, $\tilde{\mathcal{W}}_{1,2}$, $\tilde{\mathcal{W}}_{1,3}$,
$\tilde{\mathcal{W}}_{2,3}$, $\tilde{\mathcal{W}}_{1,2,3}$, $\tilde{\mathcal{W}}_{2,3,4}$, $\tilde{\mathcal{W}}$.
\qquad\=
A2.4: $\tilde{\mathcal{W}}_{1,3}$.\kill
\>
A1.1: $\mathcal{K}$, $\tilde{\mathcal{W}}_{3,4}$, $\tilde{\mathcal{W}}_{1,2,3}$, $\tilde{\mathcal{W}}$.\>
A2.1: $\tilde{\mathcal{W}}_1$, $\tilde{\mathcal{W}}_{1,3}$.\\
\>
A1.2: $\tilde{\mathcal{W}}_3$, $\tilde{\mathcal{W}}_{1,2}$, $\tilde{\mathcal{W}}_{3,4}$,
$\tilde{\mathcal{W}}_{1,2,3}$, $\tilde{\mathcal{W}}$.\>
A2.2: $\tilde{\mathcal{W}}_3$, $\tilde{\mathcal{W}}_{1,3}$.\\
\>
A1.3: $\mathcal{K}$, $\tilde{\mathcal{W}}_{3,4}$, $\tilde{\mathcal{W}}_{1,2,3}$, $\tilde{\mathcal{W}}$. \>
A2.3: $\tilde{\mathcal{W}}_{1,3}$.\\
\>
A1.4: $\tilde{\mathcal{W}}_4$, $\tilde{\mathcal{W}}_{1,2}$, $\tilde{\mathcal{W}}_{1,3}$,
$\tilde{\mathcal{W}}_{2,3}$, $\tilde{\mathcal{W}}_{i,j,k}$\,{\footnotesize{$(i<j<k)$}}, $\tilde{\mathcal{W}}$.\>
A2.4: $\tilde{\mathcal{W}}_{1,3}$.
 \end{tabbing}
Notice that for Table A2 the structures are never in the general $\tilde{\mathcal{W}}$ class
because they always satisfy the condition $\delta\oo=0$ of
$\mathcal{W}_1\oplus\mathcal{W}_2\oplus\mathcal{W}_3$ by
the $\h$-equivariance of the divergence and the module type, but in fact they all satisfy the stronger
condition of $\mathcal{W}_1\oplus\mathcal{W}_3$.
 \end{proof}

 \begin{rk}
The cocalibrated structures of Hervella-Vidal type  
$\mathcal{G}_1=\mathcal{W}_1\oplus\mathcal{W}_3\oplus\mathcal{W}_4$ are those
admitting a Hermitian connection with totally skew-symmetric torsion studied in \cite{AFS,S}. 
Our computation confirms the results of these papers about the GH-type of such structures with
the parallel torsion. 
 \end{rk}

 \section{Investigation of homogeneous models of type II}\label{S5}

The $G_2$-invariant almost complex structures on $\mathbb{S}^6$ and $\mathbb{S}^{2,4}$ (type I) are well-studied \cite{E,G,Ka,K$_2$}, type III structures are described in Appendix - the
corresponding manifolds have simple topology.
In this section we describe the almost complex models of type II
(some examples of these appeared in \cite{S}).
 % In \cite{S} the $SU(2)$ isotropy means our representation $V+\C$ and the $SO(3)$ isotropy means $ad+ad$.

Notice that by Section \ref{S4} none of these possesses a K\"ahler structure invariant with respect to
the corresponding group $G$ of dimension 9. This does not mean that they do not have K\"ahler
structures at all. We shall describe them topologically and see that in some cases such a structure exists,
however it is not $G$-invariant.

 \subsection{Homogeneous models of type II$_1$}\label{S5a}

Structures of type II$_1$ have $\m=V+\C$ as $\h$-representation.
Interpreting $\h=\mathfrak{su}(2)$ as imaginary quaternions
we can identify $V$ with the space of quaternions $\mathbb{H}$ (as the module, not algebra).
In particular, the set of $\h$-invariant complex structures is the standard unit sphere
$\mathbb{S}^2=\{q\in\Im(\mathbb{H}):q\cdot q=-1\}$.

Thus the set of $\h$-invariant (almost) complex structures
on (representation) $\m$ is $\mathcal{J}_e=\mathbb{S}^2\times\{1\}\simeq \mathbb{S}^2$.

\medskip

{\bf II$_1^a$.} The Lie group $SU(3)$ acts transitively on $\mathbb{S}^5$ with
the stabilizer $SU(2)$. Therefore we have the following diffeomorphism:
 \begin{equation*}
M=U(3)/SU(2)=U(1)\times SU(3)/SU(2)\simeq\mathbb{S}^1\times\mathbb{S}^5.
 \end{equation*}
The complex structure $J$ on this manifold is obtained from Hopf fibration
$\mathbb{S}^5\to\C P^2$ with the fiber $\mathbb{S}^1$ and the standard connection
$H\subset T\mathbb{S}^5$, so that $M$ is the $\mathbb{T}^2$ fibration over $\C P^2$.

This is the well-known Calabi-Eckmann complex structure $J_0$: at every point it is the sum of
the standard complex structure on the connection (flat CR-structure) and
a complex structure on $\mathbb{T}^2=\mathbb{S}^1\times\mathbb{S}^1$
(given by 2 parameters, which disappear if we pass to the universal cover).
Clearly, $M$ is not symplectic and so is not K\"ahler.

There is however a 2-parameter family of deformations of this structure to almost complex structure
on the universal cover $\tilde M=\R^1\times\mathbb{S}^5$. Indeed, this is given by the construction of
$(\g,\h,\m,J)$, see Tables in the Appendix. To describe $J$ fix a point, say $a=(0,0,1)\in\mathbb{S}^5$.
The CR-hyperplane $\C^2\simeq H_a\subset T_a\mathbb{S}^5$ as $\h$-representation is isomorphic to $\mathbb{H}$, and so the space of invariant structures is $\mathbb{S}^2$.
We translate this to any other point $b\in\mathbb{S}^5$ by an element of $SU(3)$ and obtain
an invariant CR-structure $J'$ on the connection $H$ (in fact, we have the trivial $\mathbb{S}^2$-bundle
of almost complex structures on the connection $H$ over $\mathbb{S}^5$). This is extended to $J=J'+J''$
by complementing $J'$ with (any) complex shift invariant structure $J''$ on the cylinder $\R^1\times\mathbb{S}^1\simeq\C/\Z$. The space of obtained structures is
$\mathcal{J}_e\simeq\mathbb{S}^2$.

Since normalization of the Lie algebra structure on $\g$ involves complex multiplication
$J_0=i$ on $\m$, there are two preferred complex structures $\pm J_0$ on this sphere.
Moving along the fiber of the Hopf fibration $\mathbb{S}^5\to\C P^2$
rotates the sphere $\mathbb{S}^2$ along the axis through these antipodal points
(the differential of this maps $\Im(\mathbb{H})\ni q\mapsto [q,i]\in T\mathbb{S}^2$).
Only $\pm J_0$ are (integrable) complex structures; for all other choices of $J$,
the Nijenhuis tensor $N_J$ is non-degenerate.

When we compactify $M$ to $\mathbb{S}^1\times\mathbb{S}^5$, we get two more parameters coming
from the torus $\mathbb{T}^2\subset U(3)/SU(2)$ (quotient of the cylinder),
as the space of complex structures on it is given by the fundamental domain
$\Sigma^2=\{|z|\ge1,|\Re(z)|\le \frac12,\Im(z)>0\}\subset\C$.
Thus the space of parameters of
$U(3)$-invariant complex structures on this compact $M^6$ is $\mathcal{J}_+=\mathcal{J}_e\times\Sigma^2$.
The latter deformation does not change integrability: $N_J$ depends only on the
first factor, and so is as described above.

\smallskip

{\bf II$_1^b$.} Similarly, the Lie group $SU(2,1)$ acts transitively on the pseudo-sphere
$N^5=\{z\in\C^3:|z_1|^2+|z_2|^2-|z_3|^2=-1\}$ with the stabilizer $SU(2)$
(inducing the action on the unit ball
$B^4\simeq\mathbb{P}N^5$ with the stabilizer $U(2)$, important for the canonical/flat
CR-structure on $\mathbb{S}^3=\p B^4$). Since $N^5\simeq \mathbb{S}^1\times B^4$, we obtain
the following diffeomorphism:
 $$
M=U(2,1)/SU(2)\simeq U(1)\times SU(2,1)/SU(2)\simeq \mathbb{T}^2\times B^4.
 $$
The invariant complex structure $J$ on this $M$, as well as on the universal cover
$\tilde M=\C\times B^4$, is however not the obvious one, since $(M,J,\oo)$ is not K\"ahler:
the structure $J'$ on the factor $B^4$ has to be $SU(2,1)$-invariant.

Again the space of moduli of all $SU(2)$-invariant structures $J$ on $\tilde M$ is
$\mathcal{J}_e\simeq \mathbb{S}^2$ and on $M$ it is $\mathcal{J}_+=\mathcal{J}_e\times\Sigma^2$.
Indeed, at every point $a\in N^5$ the natural contact space $H_a\subset T_aN^5$ is isomorphic
to $\mathbb{H}$ as $\h$-module, and the space of invariant complex structures on $H$ is the trivial
$\mathbb{S}^2$-bundle over $N^5$. This gives the complex structure $J'$, and
the shift invariant complex structure $J''$ on the factor $\mathbb{T}^2$ resp.\,$\C$
yields the structure $J=J'+J''$ on $M$ resp.\, $\tilde M$. The integrable structures among these $J$
are $J_0=\pm i$ only, for others the tensor $N_J$ is non-degenerate.

Concerning the integrability properties for the structures from $\mathcal{J}_e$
the following statement describes the generalized SKT property.

 \begin{prop}
For the compatible invariant almost symplectic form $\oo$ and the operator $d_J=d\circ J$ we have:
$d_J^2\oo\ne0$ always; $d_J^3\oo\ne0$ unless $J=J_0=\pm i$ or
$J\perp J_0\Leftrightarrow J=\cos t\cdot j+\sin t\cdot k$, $t\in\R$ (in these cases $d_J^3\oo$ vanishes);
$d_J^4\oo=0$ for any other parameter value.
 \end{prop}

Thus we obtain that not only the poles, but also the equator between them in
$\mathcal{J}_e=\mathbb{S}^2$ consists of distinguished almost complex structures.

 \subsection{Homogeneous models of type II$_2$}\label{S5b}

Structures of type II$_2$ also have $\m=V+\C$ as $\h$-representation.
For $\h=\mathfrak{su}(1,1)$, considered as imaginary split quaternions, we identify $V$ with
the module of split quaternions $\mathbb{H}_s$. The set of $\h$-invariant complex structures
$J_q(v)=v\cdot q$, $q\in\Im(\mathbb{H}_s)$ is the two sheet hyperboloid
$Z^2=\{q\in\Im(\mathbb{H}_s):q\cdot q=-1\}$.
Thus the set of $\h$-invariant (almost) complex structures on $V+\C$ is
$\mathcal{J}_h=Z^2\times\{1\}\simeq \mathbb{D}^2\times\Z_2$.

Notice also that on $V$ there are (almost) product structures $I_q(v)=v\cdot q$
forming the set $Y^2=\{q\in\Im(\mathbb{H}_s):q\cdot q=+1\}$, which is the
one sheet hyperboloid homeomorphic to the cylinder $\mathbb{S}^1\times\R^1$.

\smallskip

{\bf II$_2^a$.} The Lie group $SU(2,1)$ acts transitively on the unit pseudo-sphere in
$\C^3\simeq\R^{4+2}$ with the metric of the signature $(4,2)$:
$Q^5=\{(z_1,z_2,z_3)\in\C^3:|z_1|^2+|z_2|^2-|z_3|^2=1\}\simeq\mathbb{S}^3\times\C$.
Thus
 $$
M=U(2,1)/SU(1,1)=U(1)\times SU(2,1)/SU(1,1)\simeq \mathbb{S}^1\times\mathbb{S}^3\times\C.
 $$
To see the invariant complex structure, notice that $\mathbb{S}^1$ acts on $Q^5$, $w\mapsto e^{it}w$,
and the quotient is
 $$
Q^5/\mathbb{S}^1=\C^2\#\overline{\C P^2}\simeq\C P^2\setminus\{pt\}\simeq\C P^1\times\C;
 $$
the structure $J$ is obtained similarly to case 1 above.

Namely, the CR-connection $H\subset TQ$ carries the space $\mathcal{J}_h$ of $G$-invariant
almost complex structures $J'$, which form a trivial 2D bundle over $Q$.
Complementing this with any shift invariant complex structure $J''$ on the torus $\mathbb{T}^2$,
which is the fiber of the discussed map $M^6\to\C P^1\times\C$,
we get the almost complex structure $J=J'+J''$ on $M=U(2,1)/SU(1,1)$.
The moduli space of these structures is $\mathcal{J}_+=\mathcal{J}_h\times\Sigma^2$.

If we consider the universal cover $\hat M=\R^1\times\mathbb{S}^3\times\C$, then the
above torus becomes the cylinder $\R^1\times\mathbb{S}^1$ and (since on the cylinder
all complex structures are equivalent) the moduli space becomes not 4- but 2-dimesional,
namely $\mathcal{J}_h$.

Again there are two preferred structures, corresponding to $J'=\pm i$. Indeed, moving
along the orbit $\mathbb{S}^1$ of $U(1)$-action the space
$\mathcal{J}_h\simeq\mathbb{D}^2\times\Z_2$ rotates around its axis of symmetry through $\pm i$.
Only these two corresponding structures $J$ are integrable,
the others have non-degenerate Nijenhuis tensor $N_J$.

\smallskip

{\bf II$_2^b$.} Finally, let us discuss
 $$
M=GL(3)/SU(1,1)=GL(3)/SL(2).
 $$
It has two connected components, each being simply-connected.

We can identify $M$ with the space  $\{(v,\Pi^2,\oo)\}$, where $v\in\R^3$ is a nonzero vector,
$\Pi^2\not\ni v$ a transversal 2-plane and $\oo\in\La^2\Pi^*\setminus0$ an area form on it.
Indeed, $GL(3)$ acts transitively on this space with the stabilizer $SL(2)$. Furthermore, we
have
 $$
M^6=\R_*\times N^5,\text{ where }N^5=SL(3)/SL(2),
 $$
and $SL(2)$ is embedded into $SL(3)$ as the lower $2\times2$ block.

We identify $N^5=\{(v,p)\in\R^3\times(\R^3)^*:\langle v,p\rangle=1\}$ by choosing covector
$p$ in the annihilator of $\Pi$. At $a=(v,p)\in N$ the stabilizer $SL(2)$
acts on $T_aN=\{(w,q):\langle w,p\rangle+\langle v,q\rangle=0\}$.
This has the invariant subspace $H_a=H_a^1\oplus H_a^2$, where
$H_a^1=\{(w,0):\langle w,p\rangle=0\}$ and $H_a^2=\{(0,q):\langle v,q\rangle=0\}$.

We have: $H_a^2=(H_a^1)^*$. The $SL(2)$-invariant area form $\oo$ on $H_a^1$ is obtained
as $\iota_v\Omega$, where $\Omega$ is the volume form in $\R^3$. The invariant almost complex
structure is now given by $J_{rt}(\xi)=r\,\xi+t\,\hat\xi$, where $\xi\in H_a^1$ and
$\hat\xi=\iota_\xi\oo\in H_a^2$, $r,t\in\R$. Requirement $J_{rt}^2=-1$ implies
$J_{rt}(\hat\xi)=-\frac{1+r^2}{t}\xi-r\,\hat\xi$.
Complementing this by requirement that $J_{rt}$ maps the complement $\{(\l v,-\l p)\}\subset T_aN^5$
to $T_a\R_*$ we obtain the almost complex structure on $M$ parametrized by the pair $(r,t)$
or equivalently by the space $\mathcal{J}_h$.
Direct calculation shows that none of these structures $J_{rt}$ is integrable.
 % (but there is a preferred integrable almost product structure here instead)

The integrability properties generalizing SKT for the structures from $\mathcal{J}_h$
are given by the following statement.

 \begin{prop}
For the compatible invariant almost symplectic form $\oo$ and the operator $d_J=d\circ J$ we have:
$d_J^2\oo\ne0$ always; $d_J^3\oo\ne0$ unless $J=J_0=\pm i$
in the case $G=U(2,1)$ ($p=i$ in the case A3.5)
or $J=\cosh t\cdot i+\sinh t\cdot k$, $t\in\R$ in the case $G=GL(3)$
($p=j$ in the case A3.5; in terms of the parameters of the Table this means $q=\a\,i+\b\,k$, $\a^2-\b^2=1$)
-- only in these cases $d_J^3\oo$ vanishes; and $d_J^4\oo=0$ for any other parameter value.
 \end{prop}

Notice that for $H=SU(1,1)$ there are less invariant structures $(J,\oo)$
with $d_J^3\oo=0$ than for $H=SU(2)$.
Actually, for $G=U(2,1)$ the additional plane
$\{q\in\Im(\H_s):q\perp i\}=\langle j,k\rangle$
consists of the product structures $q^2=1$;
for $G=GL(3)$ and $p=j$ (we refer to the case A3.5 in Table A3) the plane
$\{q\in\Im(\H_s):q\perp j\}=\langle i,k\rangle$ contains the indicated
two-component curve of complex structures $q^2=-1$.

 \section{Concluding remarks}

In this paper we classified special homogeneous almost complex structures in dimension 6.
Most of them turn out to be the left-invariant structures on Lie groups and,
with some notable exceptions, most of these are solvable. Passing to quotient nil- or solv-manifolds
destroys the isotropy, so our classification easily implies

 \begin{theorem}
The only compact homogeneous almost complex manifolds in dimension 6 with semi-simple (nontrivial) isotropy are
$\mathbb{S}^6$, $\mathbb{S}^1\times\mathbb{S}^5$ and $\mathbb{S}^3\times\mathbb{S}^3$
and their finite quotients, equipped with the almost complex structures as described in Theorem \ref{Thm1}.
 \end{theorem}

 \begin{rk}
The possible quotients, that carry general parameter almost complex structures $J$ from Theorem \ref{Thm1}, are
$\mathbb{S}^1\times(\mathbb{S}^5/\Z_3)$,
$\mathbb{S}^1\tilde\times(\mathbb{S}^5/\Z_3)=(\mathbb{S}^1\times\mathbb{S}^5)/\Z_3$,
$\mathbb{RP}^3\times\mathbb{S}^3$ and $\mathbb{RP}^3\times\mathbb{RP}^3$. For some exceptional values
of parameters there are further quotients, e.g. the Calabi-Eckmann structure descends onto the
direct product %of the circle and the lens space
$\mathbb{S}^1\times(\mathbb{S}^5/\Z_n)$ for any $n$.
 \end{rk}

Some items from our list easily generalize to higher dimensions, providing examples of non-integrable
almost complex structures (with non-degenerate Nijenhuis tensor $N_J$) with abundant symmetries.
Previously some invariant almost complex structure appeared in \cite{Wo,WG}. Here are some new examples.

For every Lie algebra $\h$ and a complex vector space $(V,i)$ consider the space $\m=\h\otimes V$
with the $\h$-invariant almost complex structure given by $J(h\otimes v)=h\otimes(i\cdot v)$. Fixing a
commutative associative bilinear multiplication $Q:S^2V\to V$ we define the bracket by the formula
 $$
[h\otimes v,h'\otimes v']=[h,h']_\h\otimes Q(v,v').
 $$
With this $\m$ and $\g=\h\ltimes\m$ become Lie algebras, and the corresponding Lie group $M=G/H$
carries $G$-invariant almost complex structure $J$.

All 4 structures in Table A2 for $\h=\mathfrak{su}(2)$ resp.\ Table A4 for $\h=\mathfrak{su}(1,1)$ are
of this type with $V=\R^2$ (more precisely it corresponds to a semi-group structure on the set of
two points). This family of generalized gradings contains, in particular, the 2-step graded nilpotent algebra
$\h_1\oplus\h_2$, $[h_1,h_1']=[h,h']_2$ (and other brackets vanish). The corresponding Lie group
$M=\exp(\m)$ carries the left-invariant structure $J$ with non-degenerate Nijenhuis tensor
in the sense $N_J(\Lambda^2TM)=TM$.

Another example is the "right extension" $M^{4n+2}$ of the Heisenberg group over the quaternion space
$\H^n$. As usual we describe the corresponding Lie algebra. Let $I,J,K$ be the triple of $\op{Sp}(n)$-invariant complex structures on $\H^n$,
and $\oo_q(x,y)=g(xq,y)$ be the symplectic structure parametrized by
$q\in \mathbb{S}^2\subset\R^3(I,J,K)$, where $g(x,y)=x^t\cdot\bar y$ is the standard Hermitian structure,
$x,y\in\H^n$. Then with fixed $\ee\in\C$ we have the following brackets on $\m$ ($\varepsilon\in\{0,1\}$):
 $$
[x,y]=\oo_I(x,y)\ee,\quad [x,i\ee]=x(\tfrac12\varepsilon+\beta I),\quad [\ee,i\ee]=\varepsilon\ee.
 $$
Notice that $\H^n\oplus\langle\ee\rangle$ is the Heisenberg algebra and $\m$ its extension by derivations.
The almost complex structure $J_q(x,\ee)=(xq,i\ee)$ on $M$ is invariant with respect to the Lie group
$G=\op{Sp}(n)\ltimes \op{exp}(\m)$. The "twistor space" of this $M$
(= the bundle of moduli of the invariant almost complex structures) is the $4(n+1)$-dimensional
manifold $M^\dagger=M\times \mathbb{S}^2$ equipped with the canonical almost complex structure
 $$
J_\dagger(x,\ee,\xi)=(xq,i\ee,\xi q),\quad \xi\in T_q\mathbb{S}^2.
 $$
This structure $J_\dagger$ is non-integrable (its Nijenhuis tensor is encoded by $N_J$), but has
some natural pseudoholomorphic foliations.

Finally let us notice that for type II spaces from Theorem \ref{Thm1} the set of invariant
almost complex structures $\mathcal{J}$ is two-dimensional (if we consider simply connected model $M$):
it is parametrized by $\mathbb{S}^2$ for $\h=\mathfrak{su}(2)$
and by $\Z_2\times \mathbb{D}^2$ for $\h=\mathfrak{su}(1,1)$. The space of all these structures --- the
8-dimensional "twistor space" $M^6\times\mathcal{J}$ ---
has a natural almost complex structure, but its Nijenhuis tensor is nonzero.
For type II$_1$ spaces on $M=U(3)/SU(2)$ or $M=U(2,1)/SU(2)$ (where the space of almost complex
structures is given by 4 parameters) the corresponding "twistor space"
$M^6\times\mathcal{J}\times\Sigma^2$ is 10-dimensional and it also possesses the natural
non-integrable almost complex structure.

\vskip1cm

%\newpage

\appendix
 \section{Tables of the structure of $(\g,\h,\m,J,\omega)$}

Here we list the data for Theorems \ref{Thm1}, \ref{Thm3} and indicate some integrability properties
for the structures.
These completely encode all non-flat homogeneous almost complex manifolds with semi-simple
isotropy, only excluding the special cases $G_2^c/SU(3)=S^6$ and $G_2^*/SU(2,1)=S^{2,4}$.

\smallskip

Imaginary quaternions $\H$ are generated by $i,j,k=ij$, which anti-commute and satisfy $i^2=-1,j^2=-1$
(and hence $k^2=-1$). In tables A1,\,A2 we identify $\h=\Im(\H)$.

Imaginary split quaternions $\H_s$ are generated by $i,j,k=ij$, which anti-commute and satisfy $i^2=-1,j^2=1$
(and hence $k^2=1$). In tables A3,\,A4 we identify $\h=\Im(\H_s)$.

The brackets $[\h,\h]$ and $[\h,\m]$ are straightforward and are not included into the tables.
We include only the non-trivial brackets $[\m,\m]$.

The almost complex structure $J$ is indicated in terms of its 2 parameters for tables A1-A4
(beware that the parameter $r$ has different meaning in different tables).
$J$ has no parameters in tables A5-A6.
In the case of representation $\mathfrak{ad}^\C$ we use formula (\ref{J}) for $J$.

We list only necessary values of the Nijenhuis tensor (to minimize the tables),
the others can be restored by the rule $N_J(Jx,y)=N_J(x,Jy)=-JN_J(x,y)$;
we omit the trivial entries.

We write NDG to indicate that the tensor $N_J$ is non-degenerate;
DG$_2$ indicates that the image of $N_J:\La^2_\C TM\to TM$
is a real rank 2 subdistribution of $TM$, DG$_1$ means it is a real rank 4
subdistribution (in DG$_k$ the number $k$ is the complex codimension of $\op{Im}(N_J)$,
see \cite{K$_1$}).

The Hermitian metric is given by $g(\xi,\eta)=\oo(\xi,J\eta)$ as in Section \ref{S4}.
Thus we describe only the almost symplectic form $\oo$ and its differential.

Recall that $V$ stands for the standard representation, $\C$ for the trivial
complex and $\mathfrak{ad}$ for the adjoint representation.

\newpage

\begin{center}
{\bf Table A1:} isotropy $\mathfrak{h}=\mathfrak{su}(2)$, representation $V+\C$:
\end{center}

\hskip-2.3cm \begin{tabular}[t]{| l | l |}
    \hline
 $\mathfrak{m}$ as $\h$-representation & $\mathfrak{m}=V\oplus\C=\H\oplus \C$.
 $\mathfrak{h}$ acts from the left, $\mathfrak{ad}_{\Im\H}\subset \text{End}_\mathbb{R}(\H)$ \\
   \hline
 vectors \& parameters & $\ee\in\C$ is fixed. $x,y\in\H$ are arbitrary. $q,p,b\in\Im(\H)$,
 $\a,\b,r,\e\in\R$\\
   \hline
 structure $J$ on $\m$ & $J=(q,i)$, $q^2=-1$. $q$ acts from the right: $J(x,\ee)=(xq,i\ee)$
 \vphantom{$A^{\frac12}$}\\
   \hline
 structure $\omega$ on $\m$ & $\omega=\omega_\H+\omega_\C$, \ $\omega_\H(x,y)=\Re(x b\hs\bar y)$, \ $\omega_\C(\ee,i\ee)=1$ ($b\neq0$) \\
   \hline
 compatibility of $J$ and $\oo$ \hspace{8.5pt} & $b\in\mathbb{R}q$\\
   \hline
\end{tabular}

\vspace{1.5pt}

\hskip-3.5cm\begin{tabular}[t]{ l | l | l | l |}
   \cline{2-4}
 & Lie algebra structure of $\g$ & tensor $N_J$ \& condition for $d\omega=0$& notes
 \vphantom{$A^{\frac12}$} \\
   \cline{2-4}
A1.1 & $[x,y]=\a\Re(xi\bar y)\ee$ & \!$N_J(x,y)=\!\a(\Re(xq[q,i]\bar y)-i\Re(x[q,i]\bar y))\ee$\!  &
 $\varepsilon\in\{0,1\}$, $\a\neq0$ \vphantom{$A^{\frac12}$} \\
 & $[x,i\ee]=x(\frac{\varepsilon}{2}+ri)$& $N_J(x,\ee)=rx[q,i]$ & NDG unless $r=0$ \\
 & $[\ee,i\ee]=\varepsilon\ee$ & $d\omega=0$ iff $\varepsilon=1,b=\a i$ & \hspace{47pt} or $q=\pm i$ \\
   \cline{2-4}
A1.2 & $[x,y]=$ & $N_J(x,y)=\Re(x(q[q,i]+[q,p])\bar y)\ee+$ & DG$_2$ unless
 \vphantom{$A^{\frac12}$} \\
 & \hphantom{AA}$=(\Re(xi\bar y)+i\Re(xp\bar y))\ee$ & \quad\ $+\Re(x(q[q,p]-[q,i])\bar y)i\ee$,
 \qquad $d\omega\not=0$ & \hphantom{AAAA} $[p,q]=q[q,i]$ \\
   \cline{2-4}
A1.3 & $[x,\ee]=\alpha x$ & $N_J(x,\ee)=rx[q,i]$ & $(\varepsilon,\alpha)\in\{(0,0),$
 \vphantom{$A^{\frac12}$} \\
 & $[x,i\ee]=x(\beta+ri)$ & $d\omega=0$ iff $\alpha=\beta=0, b\in\mathbb{R}i$ &
 \qquad\qquad $(0,1),(1,0)\}$ \\
 & $[\ee,i\ee]=\varepsilon\ee$ & & DG$_1$ or NDG \\
   \cline{2-4}
A1.4 & $[x,y]=\Re(xi\bar y)\ee+\varepsilon\cdot\mathfrak{ad}_{\Im(x\bar y)}$ &
 $N_J(x,y)=(\Re(xq[q,i]\bar y)-i\Re(x[q,i]\bar y))\ee$ & \hphantom{AAA}$\varepsilon=\pm1$
 \vphantom{$A^{\frac12}$} \\
 & $[x,\ee]=3\hs\varepsilon x i$ & $N_J(x,\ee)=3\hs\varepsilon xq[q,i]$, &
 \ $\varepsilon=-1\Rightarrow\mathfrak{g}=\mathfrak{u}(3)$ \\
 & & $d\omega\not=0$ & \ $\varepsilon=1\Rightarrow\mathfrak{g}=\mathfrak{u}(2,1)$ \\
 & & & NDG unless $q=\pm i$ \\
   \cline{2-4}
\end{tabular}

\bigskip

\begin{center}
{\bf Table A2:} isotropy $\mathfrak{h}=\mathfrak{su}(2)$, representation $\mathfrak{ad}^\C$:
\end{center}

\hskip-2.3cm\begin{tabular}[t]{| l | l |}
    \hline
 $\mathfrak{m}$ as $\h$-representation & $\mathfrak{m}=\mathfrak{ad}\oplus\mathfrak{ad}=\mathfrak{ad}^\C$,
 the complex adjoint representation of $\mathfrak{h}$ \vphantom{$A^{\frac12}$} \\
    \hline
 vectors \& parameters & $x,y\in\m$. If $\m=\m_a\oplus\m_b$, we decompose $x=x_a+x_b$. \ \ $r,t\in\R$ \\
   \hline
 structure $J$ on $\m$ & $J(u,0)=(r u, t u)$, $u\in\mathfrak{ad}$ and $J^2=-1$ \vphantom{$\frac{2^2}a$}\\
    \hline
 structure $\omega$ on $\m$ & Let $K_1$, $K_2$ be the Killing forms of $\mathfrak{su}(2)$
 on each copy of $\mathfrak{ad}$ and \\
 & $J_0(u,v)=(-v,u)$, $u,v\in\mathfrak{ad}$. Then $\omega(x,y)=(K_1+K_2)(J_0x,y)$ \hspace{22.5pt} \\
    \hline
 compatibility of $J$ and $\oo$ \hspace{-1pt} & $\omega$ is always compatible, never closed\\
    \hline
\end{tabular}

\vspace{1.5pt}

\hskip-3.5cm\begin{tabular}{ l | l | l | l |}
    \cline{2-4}
 & Lie algebra structure of $\g$ & tensor $N_J$ \& condition for $d\omega=0$ & notes
 \vphantom{$A^{\frac12}$} \\
    \cline{2-4}
A2.1 & $\m=\mathfrak{su}(2)_1\oplus\mathfrak{su}(2)_2$ &
 $N_J(x_1,y_1)=-(r^2+1)[x,y]_1+t(t-2r)[x,y]_2$ \hspace{-8.5pt}
 \vphantom{$A^{\frac12}$} & $1,2$ are not gradings, \\
 & semi-simple & & NDG always \\
    \cline{2-4}
A2.2 & $\mathfrak{m}=\mathfrak{sl}_2(\mathbb{C})$ -- simple, & $N_J(x,y)=(1+r^2-t^2+2rti)[x,y]$
 \vphantom{$A^{\frac12}$} & $N_J=0$ for $J=\pm i$, \\
 & $\mathfrak{su}(2)\oplus i\,\mathfrak{su}(2)$ & & NDG else \\
    \cline{2-4}
A2.3 & $\mathfrak{m}=\mathfrak{su}(2)_0\oplus \mathfrak{su}(2)_1$ &
 $N_J(x_0,y_0)=-(r^2+1)[x,y]_0-2rt[x,y]_1$ \vphantom{$A^{\frac12}$} & Graded \\
 & $\mathfrak{su}(2)_1$ abelian.& & NDG always \\
    \cline{2-4}
A2.4 &  $\mathfrak{m}=\mathfrak{su}(2)_1\oplus \mathfrak{su}(2)_2$ &
 $N_J(x_1,y_1)=\frac{2(r^3+r)}{t}[x,y]_1+(3\hs r^2-1)[x,y]_2$ \vphantom{${\dfrac{a}2}$}
 & Graded \vphantom{$A^{\frac{8^8}2}$} \\
 & 2-step nilpotent & & NDG always \\
    \cline{2-4}
\end{tabular}

\bigskip

\pagebreak

\begin{center}
{\bf Table A3:} isotropy $\mathfrak{h}=\mathfrak{su}(1,1)$, representation $V+\C$:
\end{center}

\hskip-2.3cm\begin{tabular}[t]{| l | l |}
    \hline
 $\mathfrak{m}$ as $\h$-representation & $\mathfrak{m}=V\oplus\C=\H_s\oplus \C$.
 $\mathfrak{h}$ acts from the left, $\mathfrak{ad}_{\Im\H_s}\subset \text{End}_\mathbb{R}(\H_s)$ \\
    \hline
 vectors \& parameters & $\ee\in\C$ - fixed. $x,y\in\H_s$ - arbitrary. $q,p,u,b\in\Im(\H_s)$,
 $\a,\b,r,\e\in\R$ \\
    \hline
 structure $J$ on $\m$ & $J=(q,i)$, $q^2=-1$. $q$ acts from the right: $J(x,\ee)=(xq,i\ee)$
 \vphantom{$A^{\frac12}$} \\
    \hline
 structure $\omega$ on $\m$ & $\omega=\omega_{\H_s}+\omega_\C$, $\omega_{\H_s}(x,y)=\Re(x b\hs\bar y)$, $\omega_\C(\ee,i\ee)=1$ ($b\neq0$) \\
    \hline
 compatibility of $J$ and $\oo$ \hspace{4pt} & $b\in\mathbb{R}q$ \\
    \hline
\end{tabular}

\vspace{1.5pt}

\hskip-3.5cm\begin{tabular}[t]{ l | l | l | l |}
    \cline{2-4}
 & Lie algebra structure of $\g$ & tensor $N_J$ \& condition for $d\omega=0$ & notes
 \vphantom{$A^{\frac12}$} \\
    \cline{2-4}
A3.1 & $[x,y]=\Re(xp\bar y)\ee$& $N_J(x,y)=(\Re(xq[q,p]\bar y)-i\Re(x[q,p]\bar y))\ee$
 & $\varepsilon\in\{1,0\}$, $p^2\neq0$ \vphantom{$A^{\frac12}$} \\
 & $[x,i\ee]=x(\frac{\varepsilon}{2}+rp)$& $N_J(x,\ee)=rx[q,p]$ & NDG unless $r=0$ \\
 & $[\ee,i\ee]=\varepsilon\ee$& $d\omega=0$ iff $\e=1$, $[p,b]=p-b$ &  \hspace{46pt} or $q\in\R p$ \\
    \cline{2-4}
A3.2 & $[x,y]=\Re(xp\bar y)\ee$& $N_J(x,y)=(\Re(xq[q,p]\bar y)-i\Re(x[q,p]\bar y))\ee$
 & $\varepsilon\in\{1,0\}$, $p^2=0$ \vphantom{$A^{\frac12}$} \\
 & $[x,\ee]=rxp$& $N_J(x,\ee)=x(rq[q,p]+[q,u])$ & $[p,u]=(\varepsilon+2\a\delta^0_r)p$ \\
 & $[x,i\ee]=x(\varepsilon+\a\delta^0_r+u)$& & $p\neq0$; \ NDG unless \\
 & $[\ee,i\ee]=\varepsilon\ee$& $d\omega=0$ iff $[u,b]=p-2(\varepsilon+s\delta^0_r)b$
 & \qquad \ $[u,q]=rq[q,p]$ \\
    \cline{2-4}
A3.3 & $[x,y]=(\Re(xp\bar y)+i\Re(xu\bar y))\ee$& $N_J(x,y)=\Re(x(q[q,p]+[q,u])\bar y)\ee+$ &
 DG$_2$ if $[u,q]\neq q[q,p]$ \vphantom{$A^{\frac12}$} \hspace{-10pt} \\
 & & \hphantom{A} $+\Re(x(q[q,u]-[q,p])\bar y)i\ee$, \ \ $d\omega\not=0$ &
 \hspace{21pt} $N_J=0$ else \\
    \cline{2-4}
A3.4 & $[x,\ee]=x(\alpha+p )$& $N_J(x,\ee)=x(q[q,p]+[q,u])$&$\varepsilon=p=0$; \vphantom{$A^{\frac12}$} \\
 & $[x,i\ee]=x(\beta+u)$& &or $\varepsilon=1$, $[p,u]= p$, \\
 & $[\ee,i\ee]=\varepsilon\ee$ & $d\omega=0$ iff $[b,p]=[b,u]=0$, $\alpha=\beta=0$& \hphantom{A}
 $p^2=0$, $\alpha=0$ \\
 & & & DG$_1$ or $N_J=0$ \\
    \cline{2-4}
A3.5 & $[x,y]=\Re(x p\hs\bar y)\ee+\mathfrak{ad}_{\Im(x\bar y)}$ &
 $N_J(x,y)=(\Re(xq[q,p]\bar y)-i\Re(x[q,p]\bar y))\ee$ & \hphantom{AA} $p\in\{i,j\}$
 \vphantom{$A^{\frac12}$} \\
 & $[x,\ee]=3\epsilon\hs x p$&$N_J(x,\ee)=3\hs\varepsilon xq[q,p]$ &
    $p=i\Rightarrow \mathfrak{g}=\mathfrak{u}(2,1)$ \\
 & \quad\qquad\ $p^2=-\epsilon=\pm1$ & $d\omega\not=0$ & $p=j\Rightarrow \mathfrak{g}=\mathfrak{gl}_3$ \\
 & & & NDG unless $p=\pm q$\\
    \cline{2-4}
\end{tabular}

\vspace{10pt}

\begin{center}
{\bf Table A4:} isotropy $\mathfrak{h}=\mathfrak{su}(1,1)$, representation $\mathfrak{ad}^\C$:
\end{center}

\hskip-2.3cm\begin{tabular}[t]{| l | l |}
    \hline
 $\mathfrak{m}$ as $\h$-representation & $\mathfrak{m}=\mathfrak{ad}\oplus\mathfrak{ad}=\mathfrak{ad}^\C$,
 the complex adjoint representation of $\mathfrak{h}$ \vphantom{$A^{\frac12}$} \\
    \hline
 structure $J$ on $\m$ & $J(u,0)=(r u, t u)$, $u\in\mathfrak{ad}$ and $J^2=-1$ \vphantom{$\frac{2^2}a$} \\
    \hline
 structure $\omega$ on $\m$ & Let $K_1$, $K_2$ be the Killing forms of $\mathfrak{su}(1,1)$
 on each copy of $\mathfrak{ad}$ and \\
 & $J_0(u,v)=(-v,u)$, $u,v\in\mathfrak{ad}$. Then $\omega(x,y)=(K_1+K_2)(J_0x,y)$ \hspace{22.5pt} \\
    \hline
 compatibility of $J$ and $\oo$ \hspace{-1pt} & $\omega$ is always compatible, never closed \\
    \hline
\end{tabular}

\vspace{1.5pt}

\hskip-3.5cm\begin{tabular}{ l | l | l | l |}
    \cline{2-4}
 & Lie algebra structure of $\g$ & tensor $N_J$ \& condition for $d\omega=0$ & notes
 \vphantom{$A^{\frac12}$} \\
    \cline{2-4}
A4.1 & $\m=\mathfrak{su}(1,1)_1\oplus\mathfrak{su}(1,1)_2$ &
 $N_J(x_1,y_1)=-(r^2+1)[x,y]_1+t(t-2r)[x,y]_2$ \hspace{-8.5pt}
 \vphantom{$A^{\frac12}$} & $1,2$ are not gradings, \\
 & semi-simple & & NDG always \\
    \cline{2-4}
A4.2 & $\mathfrak{m}=\mathfrak{sl}_2(\mathbb{C})$ -- simple &
 $N_J(x,y)=(1+r^2-t^2+2rti)[x,y]$ \vphantom{$A^{\frac12}$} & $N_J=0$ for $J=\pm i$, \\
 & $\mathfrak{su}(1,1)\oplus i\,\mathfrak{su}(1,1)$ & & NDG else \\
    \cline{2-4}
A4.3 & $\mathfrak{m}=\mathfrak{su}(1,1)_0\oplus \mathfrak{su}(1,1)_1$&
 $N_J(x_0,y_0)=-(r^2+1)[x,y]_0-2rt[x,y]_1$ \vphantom{$A^{\frac12}$} & Graded \\
 & $\mathfrak{su}(1,1)_1$ Abelian & & NDG always \\
    \cline{2-4}
A4.4 & $\mathfrak{m}=\mathfrak{su}(1,1)_1\oplus \mathfrak{su}(1,1)_2$ &
 $N_J(x_1,y_1)=\frac{2(r^3+r)}{t}[x,y]_1+(3\hs r^2-1)[x,y]_2$ \vphantom{$\dfrac{a}2$} & Graded
 \vphantom{$A^{\frac{8^8}2}$} \\
 & 2-step nilpotent & & NDG always \\
    \cline{2-4}
\end{tabular}

\pagebreak

\begin{center}
{\bf Table A5:} isotropy $\mathfrak{h}=\mathfrak{sl}_2(\C)$, representation $V+\C$:
\end{center}

\hskip-2.3cm\begin{tabular}[t]{| l | l |}
    \hline
 $\mathfrak{m}$ as $\h$-representation & $\mathfrak{m}=V\oplus\C$. $V\simeq\C^2$ \vphantom{$A^{\frac12}$} \\
    \hline
 vectors \& parameters & $\ee\in\C$ is fixed. $x,y\in V$ are arbitrary. \quad $\a,\b,\gamma,r,\e\in\R$ \\
    \hline
 structure $J$ on $\m$ & $J=i$ on both submodules $V$ and $\C$.\\
    \hline
 structure $\omega$ on $\m$ & Let $\omega_0$ be the $\mathfrak{sl}_2(\C)$-invariant form 2-on $V$
 with $\C$-values. \hspace{-1pt} \\
 & Then $\omega=\omega_V+\omega_\C$, $\omega_V=\Re(\lambda \omega_0),\lambda\in\C$,
 $\omega_\C(\ee,i\ee)=1$ \\
    \hline
 compatibility of $J$ and $\oo$ \hspace{9pt} & $\omega$ is never compatible and never closed.\\
    \hline
\end{tabular}

\vspace{1.5pt}

\hskip-3.5cm\begin{tabular}{ l | l | l | l |}
    \cline{2-4}
 & Lie algebra structure of $\g$ & tensor $N_J$ \& condition for $d\omega=0$ & notes
 \vphantom{$A^{\frac12}$} \\
    \cline{2-4}
A5.1 & $[x,\ee]=\alpha x $ & $N_J=0$ & $(\varepsilon,\alpha)\in\{(0,0),$ \vphantom{$A^{\frac12}$}\\
 & $[x,i\ee]=(\beta+\gamma i)x  $ & & \qquad\qquad $(0,1),(1,0)\}$ \hspace{5pt} \\
 & $[\ee,i\ee]=\varepsilon\ee$ & & $J$ is integrable \\
    \cline{2-4}
A5.2 & $[x,y]=(\Re+r\Im)\,\omega_0(x,y)\ee$ &
 $N_J(x,y)=2(1-r)\,i\,\overline{\omega_0(x,y)}\ee$\vphantom{$\dfrac{A}{a}$} & DG$_2$ unless $r=1$\\
    \cline{2-4}

\end{tabular}

\bigskip

\begin{center}
{\bf Table A6:} isotropy $\mathfrak{h}=\mathfrak{sl}_2(\C)$, representation $\mathfrak{ad}$:\\
The only possible structure is $\m=\mathfrak{sl}_2(\C)$ and $J$ is integrable, see \S\ref{S4}.\\
There are no $\h$-invariant nonzero 2-forms on $\m$.
\end{center}

\vskip1.5cm

%%%%%%%%%%%%%%%%%%%%%%%%%%%%%%%%%%%%%%%%%%%%%%%%%%%%%%%%%%%%%%%%%%%%%%%%%%%%

 \vspace{-3pt} \hspace{-20pt} {\hbox to 12cm{ \hrulefill }}
\vspace{8pt}

{\footnotesize \hspace{-10pt} $\bullet$ Boris Kruglikov, Henrik Winther: Institute of Mathematics and
Statistics, University of Troms\o, Troms\o\ 90-37, Norway.

\hspace{-10pt} E-mails: \ \textsc{boris.kruglikov@uit.no}, \quad
\textsc{henrik.winther@uit.no}.

\hspace{-10pt} $\bullet$ Dmitri Alekseevsky: Masaryk University, Kotlarska 2, 611 37 Brno, Czech Republic;
Institute  for Information Transmission Problems, B. Karetnuj per. 19, 127994, Moscow, Russia.
\ \ E-mail: \ \textsc{daleksee@staffmail.ed.ac.uk}} \vspace{-1pt}

\end{document}